\documentclass[letterpaper, 10 pt, conference]{ieeeconf} 
\IEEEoverridecommandlockouts

\makeatletter
\def\endfigure{\end@float}
\def\endtable{\end@float}
\makeatother


\makeatletter
\@ifundefined{proof}{
  \newenvironment{proof}[1][Proof]{\par\noindent\textit{#1. }\ignorespaces}
  {\hfill$\square$\par}
}{}
\makeatother

\newenvironment{IEEEkeywords}
{\par\noindent\textbf{Keywords: }\ignorespaces}
{\par}

\usepackage[T1]{fontenc}
\usepackage{mathtools} 
\usepackage{amsfonts}
\usepackage{dsfont}
\usepackage{newtxtext,newtxmath}
\usepackage{bm}
\usepackage[final]{microtype}
\usepackage{placeins}

\usepackage[caption=false,font=footnotesize]{subfig}
\usepackage{booktabs}

\usepackage{graphicx}

\usepackage{xurl}
\usepackage{array}
\usepackage{cite}

\newtheorem{theorem}{Theorem}
\newtheorem{lemma}{Lemma}
\newtheorem{proposition}{Proposition}
\newtheorem{assumption}{Assumption}

\newtheorem{remark}{Remark}
\newtheorem{corollary}{Corollary}

\DeclareMathOperator{\proj}{proj}
\DeclareMathOperator{\col}{col}
\DeclareMathOperator{\dist}{dist}
\DeclareMathOperator{\diag}{diag}

\makeatletter
\let\DLI@thebibliography\thebibliography
\renewcommand{\thebibliography}[1]{%
  \DLI@thebibliography{#1}%
  \setlength{\itemsep}{0pt}%
  \setlength{\parskip}{0pt}%
  \setlength{\parsep}{0pt}%
}
\makeatother

\title{\LARGE \bf
Distributed Vector Assembly via Selector-Anchored Diffusion with Local Innovation}
\author{%
Diana Vieira Fernandes$^{1,2}$ and Carlos Santos Silva$^{2}$
\thanks{$^{1}$Diana Vieira Fernandes is with the Department of Engineering \& Public Policy, Carnegie Mellon University, Pittsburgh, PA, USA, and also with IN+/LARSyS, Instituto Superior T\'ecnico, Universidade de Lisboa, Lisbon, Portugal {\tt\small dianaimf@andrew.cmu.edu}}%
\thanks{$^{2}$Carlos Santos Silva is with IN+/LARSyS, Instituto Superior T\'ecnico, Universidade de Lisboa, Lisbon, Portugal {\tt\small carlos.santos.silva@tecnico.ulisboa.pt}}%
\thanks{This work was supported by Funda\c{c}\~ao para a Ci\^encia e a Tecnologia, Portugal, through the Carnegie Mellon Portugal Program under fellowship FCT 2025.00027.PRT and grants UIDB/50009/2025 and LA/P/0083/2020.}
}

\begin{document}
\maketitle
\thispagestyle{empty}
\pagestyle{empty}

\begin{abstract}
We introduce Selector-Anchored Diffusion with Local Innovation (D+LI), a distributed vector-assembly primitive that constructs complete local replicas from independently owned authoritative blocks. Orthogonal selectors preserve block identity, local innovation anchors each owner's contribution, and Laplacian diffusion propagates these contributions across the network. We show that pure averaging with raw additive selector sources cannot isolate this replica. The D+LI error operator decomposes coordinatewise as the identity minus an owner-grounded Laplacian, yielding global contraction with a factor independent of the total and block dimensions. For time-varying outer targets, we establish a finite-$T$ input-to-state stability (ISS) bound and a self-consistent estimator–controller gain condition for Lipschitz coordination maps, including non-potential ones. For semi-algebraic potential problems, we prove Kurdyka--\L{}ojasiewicz (KL) finite-length convergence of the coupled decision--assembly iteration with a fixed finite communication budget per outer iteration under explicit timescale-separation conditions. Three numerical experiments demonstrate contraction tightness and finite-horizon tracking, verify an admissible KL timescale, and illustrate discrete and non-potential outer maps.
\end{abstract}


\section{Introduction}
 Coupled decentralized estimation and control often require agent \(i\) to recover the specific remote blocks \(u_j\) entering its local update. This creates a distributed vector-assembly problem: independently owned blocks must be assembled into complete local replicas while preserving their coordinate identities and write authority. We introduce \textbf{Selector-Anchored Diffusion with Local Innovation (D+LI)} as a reusable primitive for this task.
Agent \(i\) is authoritative only for block \(u_i^k\); each inner round combines diffusion of complete local replicas with the selector innovation \(\mathbf E_i(u_i^k-\mathbf E_i^\top s_i^{k,t})\). The orthogonal selectors embed the independently owned blocks in \(u^k=\sum_{i=1}^{N}\mathbf{E}_i u_i^k\), while diffusion propagates their values to every agent. For a fixed target, the recursion converges to \(\mathbf1_N\otimes u^k\); at finite \(T\), it admits explicit contraction and moving-target bounds. D+LI was first used for decentralized constrained control of active distribution networks~\cite{fernandes_decentralized_2025}. Here it is isolated as a deterministic distributed-estimation layer, analyzed independently of that application, and then coupled to general Lipschitz coordination maps and, for full outer-loop convergence, to potential maps.

\subsection{Related Work and Positioning}
Consensus+Innovations (C+I) and related distributed estimators combine neighbor averaging with innovations from partial observations and can reconstruct a complete vector at every sensor~\cite{kar_distributed_2012,kar_consensus_2013,xie_fully_2012,sahu_mathcal_2018}. With \(H_i=\mathbf E_i^\top\) and \(z_i=u_i\), the C+I correction \(H_i^\top(z_i-H_i x_i)\) has the same algebraic form as the innovation used here. The distinction is the target and analysis: C+I estimates a common exogenous vector from distributed observation models, whereas D+LI assembles a contemporaneous decision vector endogenously from independently owned authoritative blocks. We establish its fixed-target replica, grounded-Laplacian structure, dimension-independent contraction factor, finite-\(T\) moving-target bounds, and coupling to a separate outer map. Distributed observers also reconstruct full states from partial local observations~\cite{mitra_distributed_2018}. Dynamic average consensus and shared-variable methods such as EXTRA and DIGing target time-varying aggregates or a common decision~\cite{kia_tutorial_2019,shi_extra_2015,nedic_achieving_2017,qu_harnessing_2018}, while pinning and leader-following synchronize to common references~\cite{chen_reaching_2009}. Partial-decision NE/GNE methods are the closest optimization-side analogue because agents maintain estimates of the joint decision profile~\cite{pavel_distributed_2020,bianchi_fully_2021,bianchi_fast_2022}; however, those estimates are embedded in and analyzed jointly with the equilibrium-seeking iteration. D+LI instead isolates vector assembly as an inner distributed-estimation primitive. Feedback optimization and ADMM address physical feedback and constraint decomposition rather than this assembly task~\cite{erseghe_distributed_2014,dallanese_distributed_2013,picallo_closing_2020}. Table~\ref{tab:comparison} in Appendix~\ref{app:comparison} provides a detailed architectural comparison.

Our contributions are summarized as follows:
\begin{itemize}
    \setlength{\topsep}{1pt}
    \setlength{\itemsep}{0pt}
    \setlength{\parsep}{0pt}
    \setlength{\partopsep}{0pt}
    \item We formulate Selector-Anchored D+LI as a multi-source, selector-partitioned vector-assembly primitive and show why raw additive selector sources under pure averaging cannot isolate the assembled vector (Proposition~\ref{prop:no_avg}).
    \item We prove global contraction to the fixed-target replica \(\mathbf s^\star=\mathbf1_N\otimes u^k\). Coordinatewise, the operator is the identity minus an owner-grounded Laplacian, yielding a contraction factor independent of the total and block dimensions (Lemma~\ref{lem:contraction}, Remark~\ref{rem:subspace}).
    \item We establish finite-\(T\) ISS and tracking bounds under outer-target variation, together with a self-consistent gain bound requiring only a Lipschitz coordination map on a compact set; potentiality and monotonicity are unnecessary for these estimator results (Theorem~\ref{thm:iss_2ts}, Corollary~\ref{cor:tracking}, Theorem~\ref{thm:small_gain_2ts}).
    \item For potential maps with semi-algebraic \(F\), we prove finite-length KL convergence of the coupled projected-gradient/D+LI scheme under an explicit timescale-separation condition (Theorem~\ref{thm:kl_conf}).
\end{itemize}

\section{Problem Formulation}
Consider \(N\ge2\) agents on a communication graph \(\mathcal G=(\mathcal V,\mathcal E)\). 
\begin{assumption}[Communication graph]
\label{ass:graph}
The graph \(\mathcal G\) is fixed, connected, and undirected.
\end{assumption}
Let \(L\) denote its Laplacian and \(\mathcal N_i:=\{j:(i,j)\in\mathcal E\}\) the neighbor set of agent \(i\). The global decision vector is block-partitioned as \(u=[u_1^\top,\dots,u_N^\top]^\top\in\mathbb{R}^d\), where \(u_i\in\mathbb{R}^{d_i}\) is controlled by agent \(i\), and \(\sum_i d_i=d\). The feasible set is \(\mathcal U=\mathcal U_1\times\cdots\times\mathcal U_N\), where each \(\mathcal U_i\) is the local feasible set of agent \(i\). Agent \(i\) has exclusive write authority over its block through the selector \(\mathbf E_i\), whose partition identities are given in \eqref{eq:selector_partition}. Local costs are $f_i:\mathbb{R}^d\to\mathbb{R}$, and coupled constraints $c(u)\le0$, with $c:\mathbb{R}^d\to\mathbb{R}^m$, are penalized by $\phi(c):=\|[c]_+\|_2^2$ with weight $\rho>0$.

We study a decentralized coordination problem specified by local block maps \(g_i:\mathbb R^d\to\mathbb R^{d_i}\). One important instance is the penalized-gradient map
\begin{equation}\label{eq:coord_map_local}
g_i(u):=\nabla_{u_i}f_i(u)
+\rho\,\nabla_{u_i}\phi(c(u)),
\end{equation}
where \(f_i\) and \(c\) are continuously differentiable. The estimator and gain results below apply to any block maps satisfying Assumption~\ref{ass:smooth_c}. The global coordination map assembles these blocks as
\begin{equation}\label{eq:coord_map}
g(u):=\sum_{i=1}^N \mathbf{E}_i\, g_i(u)\in\mathbb{R}^d.
\end{equation}
For a stepsize \(\eta>0\), the coordination problem is to find \(u^\star\in\mathcal U\) satisfying the projected fixed-point condition
\begin{equation}\label{eq:coord_fp}
u^\star=\proj_{\mathcal{U}}\!\big(u^\star-\eta\,g(u^\star)\big),
\end{equation}
equivalently the generalized stationarity condition
\begin{equation}\label{eq:coord_ge}
0\in g(u^\star)+N_{\mathcal{U}}(u^\star),
\end{equation}
where $N_{\mathcal{U}}(\cdot)$ is the normal cone and $\proj_{\mathcal{U}}$ acts blockwise.

\begin{remark}[Potential special case]\label{rem:pot_special}
If there exists a scalar potential $F$ such that $\nabla_{u_i}F(u)=g_i(u)$ for all $i$, then \eqref{eq:coord_ge} reduces to the first-order stationarity condition of projected gradient on $F$. The contraction, ISS, and estimator--controller gain analyses do not require this potential structure. The outer-loop convergence analysis of Section~\ref{sec:kl} specializes to the potential case under Assumption~\ref{ass:kl}.
\end{remark}

\begin{assumption}[Feasible-set and block-map regularity]
\label{ass:smooth_c}
For every \(i\), the set \(\mathcal U_i\subset\mathbb R^{d_i}\) is nonempty, compact, and convex, and
\begin{align}
\mathcal U:=\mathcal U_1\times\cdots\times\mathcal U_N.
\end{align}
For every \(i\), the block map \(g_i:\mathbb R^d\to\mathbb R^{d_i}\) satisfies, for some \(L_i\ge0\),
\begin{align}
\|g_i(x)-g_i(y)\|_2
\le L_i\|x-y\|_2,
\qquad \forall x,y\in\mathcal U.
\end{align}
Define \(L_{\mathrm{blk}}:=\max_i L_i\). Each agent knows the block sets \(\{\mathcal U_j\}_{j=1}^N\), or equivalently their projection operators, so that it can evaluate the blockwise projection \(\proj_{\mathcal U}\).
\end{assumption}
By selector orthogonality,
\begin{align}
\operatorname{Lip}_{\mathcal U}(g)
\le
\bar L_g
:=
\left(\sum_{i=1}^N L_i^2\right)^{1/2}
\le
\sqrt N\,L_{\mathrm{blk}}.
\end{align}
Thus \(\bar L_g\) is always an admissible global Lipschitz constant. Hereafter, \(L_g\) denotes any certified Lipschitz constant of \(g\) on \(\mathcal U\), including a sharper structure-specific constant when available. The derivation is given in Appendix~\ref{app:estimator_gain_proofs}.
Consequently, \(g\) is continuous on the compact set \(\mathcal U\), and \(G_{\max}:=\max_{u\in\mathcal U}\|g(u)\|_2<\infty\).
\paragraph*{Partial Observability and Coupling.} 
The main challenge is \textit{partial observability of the coupled decision vector}. While agent $i$ has actuation authority over $u_i$, its local block gradient $\nabla_{u_i} f_i(u)$ and constraint evaluation $c(u)$ generally depend on the full vector $u$. For example, in AC power flow, the voltage magnitude at node $i$ is a nonlinear function of power injections at all other nodes $j \neq i$. To perform a decentralized projected update $u_i^+=\proj_{\mathcal U_i}\!\big(u_i-\eta\, g_i(u)\big)$, agent $i$ requires knowledge of $u_j$ for all $j$. Since centralized communication is unavailable, agent $i$ must maintain a local estimate $s_i \in \mathbb{R}^d$ of the global vector $u$. The objective is to design a distributed update law such that $s_i \to u$ strictly through neighbor-to-neighbor communication.

\section{The Selector-Anchored D+LI Dynamics}
\label{sec:Selector-Anchored}
To solve the estimation problem under partial observability, we introduce the Selector-Anchored D+LI dynamics. 
\subsection{Definitions}
Let $s_i^{k,t}\in\mathbb{R}^d$ denote agent $i$'s estimate of $u^k$ at inner round $t$ (outer iteration $k$). We define the \textbf{Selector Matrix} $\mathbf{E}_i \in \{0,1\}^{d \times d_i}$ as the block-column of the identity matrix corresponding to the indices of $u$ controlled by agent $i$.
\paragraph{Selector partition.}
Because each \(\mathbf E_i\) is the block-column of \(I_d\) associated with \(u_i\), the selector matrices satisfy
\begin{equation}\label{eq:selector_partition}
\mathbf{E}_i^\top \mathbf{E}_i = I_{d_i},
\quad
\mathbf{E}_i^\top \mathbf{E}_j = \mathbf{0}\quad \forall i\neq j,
\quad
\sum_{i=1}^N \mathbf{E}_i \mathbf{E}_i^\top = I_d.
\end{equation}
This property implies that the global state space is perfectly partitioned among agents, with no overlap in actuation authority. Define the assembled global decision/state vector at outer iteration $k$:
\begin{align}\label{eq:assembled_vector}
u^k
:=\sum_{i=1}^N\mathbf E_i u_i^k\in\mathbb R^d,
\quad \mathbf E_i^\top u^k
=u_i^k
\quad\text{(by \eqref{eq:selector_partition})}.
\end{align}
\subsection{Update Law (inner loop over $t$; outer loop over $k$)}
At outer iteration $k$, let $u^{k}$ denote the assembled vector defined in \eqref{eq:assembled_vector}, and let $s_i^{k,t}\in\mathbb{R}^d$ denote agent $i$'s estimate at inner round $t$. The inner update is
\begin{align}\label{eq:DLI_update_inner}
s_i^{k,t+1}
&=
s_i^{k,t}
-\alpha\sum_{j\in\mathcal N_i}\big(s_i^{k,t}-s_j^{k,t}\big)
+\mathbf{E}_i\Big(u_i^{k} - \mathbf{E}_i^\top s_i^{k,t}\Big).
\end{align}
Define the symmetric diffusion matrix $W:=I_N-\alpha L$. Define $\mathbf{s}^{k,t}:=\col(s_1^{k,t},\dots,s_N^{k,t})\in\mathbb{R}^{Nd}$ and $\mathbf{P}:=\mathrm{blkdiag}(\mathbf{P}_1,\dots,\mathbf{P}_N)$ with $\mathbf{P}_i=\mathbf{E}_i\mathbf{E}_i^\top$. The stacked inner update reads
\begin{equation}\label{eq:stacked_update_inner}
\mathbf{s}^{k,t+1}=(W\otimes I_d)\mathbf{s}^{k,t}
+\mathbf{P}\Big((\mathbf{1}_N\otimes u^{k})-\mathbf{s}^{k,t}\Big).
\end{equation}
Each inner round communicates a \(d\)-dimensional replica per edge. The resulting fixed-memory linear iteration supports warm starts and admits the finite-\(T\) moving-target bound of Theorem~\ref{thm:iss_2ts}.
\paragraph{Warm start between outer iterations.}
We carry estimates across outer steps:
\begin{equation}\label{eq:warmstart}
\mathbf{s}^{k+1,0}:=\mathbf{s}^{k,T}.
\end{equation}
Let \(\mathbf H:=\mathbf 1_N\otimes I_d\). The stacked source term is
\begin{align}
\mathbf b(u^k)
&:=\mathbf P\mathbf H u^k \notag\\
&=
\col\!\left(
\mathbf E_1u_1^k,\ldots,\mathbf E_Nu_N^k
\right).
\end{align}
Thus the recursion can be written as
\begin{align}
\mathbf s^{k,t+1}
=
\mathbf A_{\rm sys}\mathbf s^{k,t}
+\mathbf b(u^k),
\qquad
\mathbf A_{\rm sys}
:=
(W\otimes I_d)-\mathbf P.
\end{align}
The \(N\) source channels act on mutually orthogonal coordinate subspaces, and the identity \(\sum_i\mathbf E_i\mathbf E_i^\top=I_d\) guarantees complete coverage of the assembled vector. Hence the fixed point is generated collectively by the distributed block sources; no individual agent possesses or injects the complete target. Note that agent $i$ does not need the full $u^{k}$: since $\mathbf{E}_i^\top u^{k}=u_i^{k}$, the innovation term uses only local truth.

\paragraph{Local innovation (write access).}
The innovation term $\mathbf{E}_i(u_i^k-\mathbf{E}_i^\top s_i^{k,t})$ \emph{anchors} the locally owned coordinates by injecting the instantaneous substitution error on agent $i$'s block. Note that diffusion and innovation act in the same inner step; hence, the locally owned coordinates are not set exactly to $u_i^k$ at each round, but the combined mapping contracts the global stacked error (Lemma~\ref{lem:contraction}).

\paragraph{Diffusion (read access).}
Let $\mathcal I_i\subset\{1,\dots,d\}$ denote the coordinate set owned by agent $i$ (columns selected by $\mathbf{E}_i$), and $\mathcal I_i^c:=\{1,\dots,d\}\setminus\mathcal I_i$ its complement. For any coordinate $\ell\in\mathcal I_i^c$, the innovation term does not act on entry $\ell$, and the update reduces to Laplacian diffusion on that coordinate:
\begin{align}
\big[s_i^{k,t+1}\big]_\ell
=
\big[s_i^{k,t}\big]_\ell
-\alpha \sum_{j\in\mathcal N_i}\Big(\big[s_i^{k,t}\big]_\ell-\big[s_j^{k,t}\big]_\ell\Big),
\quad \ell\in\mathcal I_i^c.
\end{align}
Hence, remote values propagate through the network and are replicated asymptotically.

\begin{remark}[Communication cost]\label{rem:comm_cost}
Each inner round exchanges the full $d$-dimensional estimate per edge, versus the shared-variable dimension in shared-variable consensus; exact replication is also achievable by flooding in $\mathrm{diam}(\mathcal G)$ rounds. D+LI trades these alternatives for a routing-stateless linear iteration: no message logs or routing state, warm-started across outer iterations, with ISS to outer variation (Theorem~\ref{thm:iss_2ts}) quantifying the accuracy of the finite-$T$ replica under a moving target.
\end{remark}

\section{Stability and ISS Analysis}\label{sec:Stability_ISS}

\subsection{Error Dynamics (inner contraction; outer disturbance)}
Define the inner-loop estimation error
\begin{align}
e_i^{k,t}:=s_i^{k,t}-u^{k},\qquad
\mathbf{e}^{k,t}:=\col(e_1^{k,t},\dots,e_N^{k,t})\in\mathbb{R}^{Nd}.
\end{align}
From \eqref{eq:stacked_update_inner}, the inner error dynamics are
\begin{equation}\label{eq:error_dynamics_inner}
\mathbf{e}^{k,t+1}=\underbrace{\Big((W\otimes I_d)-\mathbf{P}\Big)}_{\mathbf{A}_{\mathrm{sys}}}\mathbf{e}^{k,t}.
\end{equation}
Across outer iterations, using the warm start \eqref{eq:warmstart} and $\Delta u^{k}:=u^{k+1}-u^{k}$,
\begin{equation}\label{eq:error_outer}
\mathbf{e}^{k+1,0}=\mathbf{s}^{k,T}-(\mathbf{1}_N\otimes u^{k+1})
=\mathbf{e}^{k,T}-(\mathbf{1}_N\otimes \Delta u^{k}).
\end{equation}

\subsection{Inner Loop Contraction}
Standard average consensus contracts the disagreement subspace while leaving the consensus subspace invariant. In D+LI, the selector-partitioned innovations also remove these invariant consensus modes because the distributed source subspaces jointly cover \(\mathbb R^d\).

\begin{lemma}[Global Contraction]\label{lem:contraction}
Under Assumption~\ref{ass:graph}, the Laplacian \(L=L^\top\) has eigenvalues \(0=\lambda_1(L)<\lambda_2(L)\le\cdots\le\lambda_N(L)\).
Let $\mathbf{P}=\mathrm{blkdiag}(\mathbf{P}_1,\dots,\mathbf{P}_N)$ with $\mathbf{P}_i=\mathbf{E}_i\mathbf{E}_i^\top$. By the selector-partition identity \eqref{eq:selector_partition}, $\sum_{i=1}^N \mathbf{P}_i = I_d$. Define $W := I_N - \alpha L$ and
\begin{align}
\mathbf{A}_{\mathrm{sys}} \;:=\; W\otimes I_d \;-\; \mathbf{P}.
\end{align}
If $0<\alpha<1/\lambda_N(L)$, then $\|\mathbf{A}_{\mathrm{sys}}\|_2<1$. Consequently, the mapping is a contraction in the Euclidean norm:
\begin{align}
\|\mathbf{A}_{\mathrm{sys}}\mathbf{x}\|_2 \le \sigma \|\mathbf{x}\|_2,\qquad \forall \mathbf{x}\in\mathbb{R}^{Nd},
\end{align}
with $\sigma:=\|\mathbf{A}_{\mathrm{sys}}\|_2\in[0,1)$.
\end{lemma}
The proof is given in Appendix~\ref{app:estimator_gain_proofs}.

\begin{proposition}[Failure of additive-source averaging]
\label{prop:no_avg}
Under Assumption~\ref{ass:graph}, let \(W=I_N-\alpha L\) with \(\alpha>0\). Consider
\begin{align}
\mathbf s^{t+1}
&=
(W\otimes I_d)\mathbf s^t+\mathbf b(u), \notag\\
\mathbf b(u)
&=
\col(\mathbf E_1u_1,\ldots,\mathbf E_Nu_N).
\end{align}
If \(u\neq0\), this recursion admits no fixed point. If \(u=0\), its fixed-point set is the consensus subspace \(\{\mathbf1_N\otimes v:v\in\mathbb R^d\}\). Consequently, \(\mathbf1_N\otimes u\) is never an isolated fixed point of this pure-averaging recursion.
\end{proposition}
The proof is given in Appendix~\ref{app:estimator_gain_proofs}.
\begin{remark}[Coordinate-subspace spectrum]\label{rem:subspace}
Let \(\chi_i\in\mathbb R^N\) denote the \(i\)-th standard basis vector. Each coordinate subspace is invariant under \(\mathbf A_{\mathrm{sys}}\), where the operator acts as an owner-grounded matrix \(W-\chi_i \chi_i^\top\). Hence \(\sigma = \max_{i:d_i\ge1} \left\|W-\chi_i \chi_i^\top\right\|_2\). Thus, for fixed owning agents, the contraction factor is independent of the total and block dimensions, although each round still exchanges a \(d\)-dimensional replica. The spectral decomposition is given in Appendix~\ref{app:estimator_gain_proofs}.
\end{remark}
\subsection{Input-to-State Stability (ISS)}
We treat the outer-step change $\Delta u^{k}:=u^{k+1}-u^{k}$ as an external disturbance to the estimator across outer iterations. For fixed $u^{k}$, the unique fixed point of \eqref{eq:stacked_update_inner} is $\mathbf{s}^\star=\mathbf{1}_N\otimes u^{k}$, i.e., $s_i^\star=u^{k}$ for all $i$.
\begin{theorem}[Finite-$T$ contraction and outer-step ISS]\label{thm:iss_2ts}
Under the hypotheses of Lemma~\ref{lem:contraction}, let $\sigma:=\|\mathbf{A}_{\mathrm{sys}}\|_2\in[0,1)$. Then, for every integer $T\ge1$, the inner loop contracts as
\begin{equation}\label{eq:inner_contr_T}
\|\mathbf{e}^{k,T}\|_2 \le \sigma^{T}\,\|\mathbf{e}^{k,0}\|_2.
\end{equation}
Moreover, the outer-step propagation satisfies
\begin{equation}\label{eq:outer_iss}
\|\mathbf{e}^{k+1,0}\|_2
\le \sigma^{T}\,\|\mathbf{e}^{k,0}\|_2 + \sqrt{N}\,\|\Delta u^{k}\|_2.
\end{equation}
\end{theorem}
The proof is given in Appendix~\ref{app:estimator_gain_proofs}.
\subsection{Self-Consistent Gain Bound for the Decentralized Interconnection}\label{subsec:small_gain}
Fix an integer \(T\ge1\) and a stepsize \(\eta>0\). Initialize
\(u^0\in\mathcal U\) and \(\mathbf s^{0,0}\in\mathbb R^{Nd}\).
\paragraph{Projected estimate (for gradient evaluation).}
Define
\begin{equation}\label{eq:proj_estimate}
\bar s_i^{k,T} := \proj_{\mathcal U}\!\big(s_i^{k,T}\big),
\end{equation}
where $\proj_{\mathcal U}$ acts blockwise.
\paragraph{Decentralized outer update.}
Each agent updates only its local block using its assembled copy:
\begin{equation}\label{eq:outer_update_local}
u_i^{k+1}
=
\proj_{\mathcal{U}_i}\!\Big(
u_i^{k}-\eta\, g_i\!\left(\bar s_i^{k,T}\right)
\Big).
\end{equation}
Here \(g_i\) denotes the local block map specified in the problem formulation and satisfying Assumption~\ref{ass:smooth_c}; \eqref{eq:coord_map_local} gives the penalized-gradient instance. The assembled vector evolves as
\begin{equation}\label{eq:outer_update_assembled}
\tilde u^{k+1} := u^{k} - \eta \sum_{i=1}^N \mathbf{E}_i\, g_i\!\left(\bar s_i^{k,T}\right),
\qquad
u^{k+1} := \proj_{\mathcal{U}}(\tilde u^{k+1}),
\end{equation}
where $\proj_{\mathcal U}$ acts blockwise. Under Assumption~\ref{ass:smooth_c}, each $g_i$ is $L_i$-Lipschitz with $L_{\mathrm{blk}}:=\max_i L_i$. Since $u_i^k\in\mathcal U_i$ for all $i$ (by \eqref{eq:outer_update_local}) and $\mathcal U=\prod_i \mathcal U_i$, we have $u^k\in\mathcal U$.
\paragraph{Outer-loop-to-estimator gain.}
Define \(\hat g^{k}:=\sum_{i=1}^N \mathbf{E}_i\, g_i(\bar s_i^{k,T})\). Orthogonality of the selector partition, Lipschitz continuity, and nonexpansiveness of projection give
\begin{align}
\|\hat g^{k}-g(u^{k})\|_2
&\le L_{\mathrm{blk}}\|\mathbf e^{k,T}\|_2,
\notag\\
\|\Delta u^{k}\|_2
&\le \eta\big(\|g(u^{k})\|_2 + L_{\mathrm{blk}}\|\mathbf{e}^{k,T}\|_2\big).
\label{eq:outer_gain_bounds}
\end{align}
The derivation is given in Appendix~\ref{app:estimator_gain_proofs}.
\begin{theorem}
\label{thm:small_gain_2ts}
\textnormal{\textit{(Self-consistent estimator--controller gain\\
bound)}}\par\nobreak
Under the hypotheses of Lemma~\ref{lem:contraction} and Assumption~\ref{ass:smooth_c}, let \(G_{\max}:=\max_{u\in\mathcal U}\|g(u)\|_2\). If
\begin{equation}\label{eq:small_gain_2ts}
q_T
:=
\sigma^T\bigl(1+\sqrt N\,\eta L_{\mathrm{blk}}\bigr)
<1,
\end{equation}
then the estimator recursion satisfies
\begin{equation}
\|\mathbf e^{k+1,0}\|_2
\le
q_T\|\mathbf e^{k,0}\|_2
+
\sqrt N\,\eta G_{\max}.
\end{equation}
Consequently,
\begin{equation}
\limsup_{k\to\infty}\|\mathbf e^{k,0}\|_2
\le
\frac{\sqrt N\,\eta G_{\max}}{1-q_T}.
\end{equation}
Moreover, 
\begin{equation}\label{eq:small_gain_terminal_limsup}
\limsup_{k\to\infty}
\|\mathbf e^{k,T}\|_2
\le
\frac{
\sqrt N\,\eta G_{\max}\sigma^T
}{
1-q_T
}.
\end{equation}
\end{theorem}
The proof is given in Appendix~\ref{app:estimator_gain_proofs}.
The gain result requires only Lipschitz continuity of \(g\) on the compact set \(\mathcal U\), without potentiality or monotonicity. Outer-loop convergence requires additional structure and is treated in Section~\ref{sec:kl} for the potential case.

\begin{corollary}[Tracking bound under bounded variation]\label{cor:tracking} 
Under the hypotheses of Lemma~\ref{lem:contraction}, let $\sigma:=\|\mathbf A_{\mathrm{sys}}\|_2\in[0,1)$, fix an integer $T\ge1$, and assume that $\|\Delta u^k\|_2\le \bar\Delta$ for all $k$. Then the estimator error is ultimately bounded as \(\limsup_{k\to\infty}\|\mathbf{e}^{k,0}\|_2 \le \frac{\sqrt{N}}{1-\sigma^{T}}\bar\Delta\), \(\limsup_{k\to\infty}\|\mathbf{e}^{k,T}\|_2 \le \frac{\sqrt{N}\sigma^{T}}{1-\sigma^{T}}\bar\Delta\).
\end{corollary}
The proof is given in Appendix~\ref{app:estimator_gain_proofs}.
\section{KL-Based Convergence}
\label{sec:kl}
We now prove finite-length convergence of the joint iterate $\xi^k:=(u^k,\mathbf s^{k,0})$ when D+LI couples to the projected-gradient outer loop \eqref{eq:outer_update_local}--\eqref{eq:outer_update_assembled}.

\begin{assumption}[Potential and KL regularity]
\label{ass:kl}
There exists a continuously differentiable semi-algebraic function \(F:\mathbb R^d\to\mathbb R\) such that
\(\nabla F(u)=g(u)\) for every \(u\in\mathcal U\).
Moreover, \(\mathcal U\) is semi-algebraic.
\end{assumption}

\paragraph{Lyapunov function.}
Recall that \(\mathbf H=\mathbf 1_N\otimes I_d\), and define
\begin{align}
r^k
&:=\mathbf s^{k,0}-\mathbf H u^k=\mathbf e^{k,0},
\notag\\
\Delta u^k
&:=u^{k+1}-u^k,
\notag\\
\Delta\mathbf s^k
&:=\mathbf s^{k+1,0}-\mathbf s^{k,0}.
\end{align}

By the warm start \eqref{eq:warmstart}, \(\mathbf s^{k+1,0}=\mathbf s^{k,T}\), hence
\begin{align}
\Delta \mathbf s^k = \mathbf s^{k,T}-\mathbf s^{k,0}
= \mathbf e^{k,T}-\mathbf e^{k,0}.
\end{align}
Let \(A_T:=\mathbf A_{\mathrm{sys}}^{\,T}\) and \(a:=\sigma^T<1\). Since \(\mathbf e^{k,T}=A_T\mathbf e^{k,0}\) and \(\Delta\mathbf s^k=(A_T-I)\mathbf e^{k,0}\), symmetry of \(A_T\) and the warm-start identity give
\begin{align}
\|\mathbf e^{k,T}\|_2
&\le \frac{a}{1-a}\|\Delta\mathbf s^k\|_2,
\label{eq:eT_deltas_bound}\\
r^{k+1}
&=\mathbf e^{k,T}-\mathbf H\Delta u^k .
\notag
\end{align}
Define
\begin{equation}\label{eq:Psi_conf}
\Psi(\xi) := F(u) 
+ \frac{\kappa_s}{2}\|\mathbf{s} - \mathbf{1}_N\otimes u\|_2^2 
+ \iota_{\mathcal{U}}(u),
\end{equation}
where $\kappa_s>0$ weights the assembly-tracking penalty and $\iota_{\mathcal U}$ is the indicator of the feasible set. Since $F$ is semi-algebraic and the remaining terms are polynomial or indicator functions of semi-algebraic sets, $\Psi$ is semi-algebraic and satisfies the Kurdyka--\L{}ojasiewicz (KL) property. Define the gradient mismatch \(\delta^k:=\hat g^k-\nabla F(u^k)\). From Section~\ref{subsec:small_gain}, \(\|\delta^k\|_2\le L_{\mathrm{blk}}\|\mathbf e^{k,T}\|_2\).

\begin{theorem}[Perturbed one-step descent]
\label{thm:suffdec_conf}
Under the hypotheses of Lemma~\ref{lem:contraction} and Assumptions~\ref{ass:smooth_c} and~\ref{ass:kl}, fix
\begin{align}
T\ge1,\quad
\eta>0,\quad
\varepsilon>0,\quad
\tau>0,\quad
\kappa_s>0,
\end{align}
where \(T\) is an integer. Define
\begin{align}
C_e&:=\frac{\varepsilon}{2}L_{\mathrm{blk}}^2,\quad
C_u:=\frac{1}{2\varepsilon},\notag\\
\tilde c_u
&:=\eta^{-1}-\frac{L_g}{2}-C_u
-\frac{\kappa_s}{2}(1+\tau^{-1})N,\notag\\
c_s
&:=\frac{\kappa_s}{2}
\frac{1-\sigma^T}{1+\sigma^T}, \quad
\tilde C_e :=C_e+\frac{\kappa_s\tau}{2}.
\end{align}

If \(\tilde c_u>0\), then one outer iteration yields
\begin{align}
\Psi(\xi^{k+1})
\le{}&
\Psi(\xi^k)
-\tilde c_u\|u^{k+1}-u^k\|_2^2 \notag\\
&-c_s\|\mathbf{s}^{k+1,0}-\mathbf{s}^{k,0}\|_2^2
+\tilde C_e\|\mathbf e^{k,T}\|_2^2 .
\end{align}
\end{theorem}
The proof is given in Appendix~\ref{app:descent_proofs}.

\begin{corollary}[Pointwise sufficient decrease]\label{cor:pointwise_descent}
Under the hypotheses and notation of Theorem~\ref{thm:suffdec_conf}, suppose that
\begin{equation}\label{eq:pointwise_absorb}
c_s^{+}(T)
:=
c_s
-
\tilde C_e\left(\frac{\sigma^T}{1-\sigma^T}\right)^2
>0 .
\end{equation}
Then
\begin{align}
\Psi(\xi^{k+1})
\le
\Psi(\xi^k)
-
\tilde c_u\|\Delta u^k\|_2^2
-
c_s^{+}(T)\|\Delta\mathbf s^k\|_2^2 .
\end{align}
\end{corollary}
The proof is given in Appendix~\ref{app:descent_proofs}.

For fixed parameters satisfying \(\tilde c_u>0\), \(c_s^{+}(T)\to\kappa_s/2>0\) as \(T\to\infty\). Thus sufficiently many inner rounds ensure pointwise sufficient decrease. Hereafter, \(\partial\Psi\) denotes the limiting subdifferential. For \(\xi=(u,\mathbf s)\) with \(u\in\mathcal U\),
\begin{align}
\partial\Psi(\xi)
=
\begin{bmatrix}
\nabla F(u)
+\kappa_s\mathbf H^\top(\mathbf H u-\mathbf s)
+N_{\mathcal U}(u)
\\[1mm]
\kappa_s(\mathbf s-\mathbf H u)
\end{bmatrix}.
\label{eq:Psi_subdiff}
\end{align}

\begin{lemma}[Relative error]
\label{lem:relerr_conf}
Under the hypotheses of Lemma~\ref{lem:contraction} and Assumptions~\ref{ass:smooth_c} and~\ref{ass:kl}, fix an integer \(T\ge1\) and \(\eta>0\). Then there exists \(C>0\) such that
\begin{align}
\dist(0,\partial\Psi(\xi^{k+1}))
\le
C\Big(
\|u^{k+1}-u^k\|_2
+
\|\mathbf{s}^{k+1,0}-\mathbf{s}^{k,0}\|_2
\Big).
\end{align}
\end{lemma}
The proof is given in Appendix~\ref{app:relerr_proof}.

\begin{theorem}[KL convergence in the potential case]
\label{thm:kl_conf}
Fix an integer \(T\ge1\) and parameters \(\eta>0, \varepsilon>0, \tau>0, \kappa_s>0.\)
Under the hypotheses of Lemma~\ref{lem:contraction} and Assumptions~\ref{ass:smooth_c} and~\ref{ass:kl}, suppose that
\(\tilde c_u>0, \quad c_s^{+}(T)>0,\) where the constants are defined in Theorem~\ref{thm:suffdec_conf} and \eqref{eq:pointwise_absorb}. Then the sequence \(\{\xi^k\}=\{(u^k,\mathbf s^{k,0})\}\) satisfies:
\begin{enumerate}
\setlength{\topsep}{1pt}
\setlength{\itemsep}{0pt}
\setlength{\parsep}{0pt}
\setlength{\partopsep}{0pt}
\item \textbf{Finite length:}
\(\sum_{k=0}^\infty\|\xi^{k+1}-\xi^k\|_2<\infty\).
\item \textbf{Convergence:}
\(\xi^k\to\xi^\star=(u^\star,\mathbf s^\star)\),
a critical point of \(\Psi\).
\item \textbf{Stationarity:}
\(0\in g(u^\star)+N_{\mathcal U}(u^\star)\).
\item \textbf{Exact assembly agreement:}
\(\mathbf s^\star=\mathbf 1_N\otimes u^\star\).
\end{enumerate}
\end{theorem}

\begin{proof}
Sufficient decrease and compactness of \(\mathcal U\) bound the joint iterates. Relative error, continuity on the domain, and the KL property yield finite length and convergence~\cite{attouch_convergence_2013}. Equation~\eqref{eq:Psi_subdiff} gives exact assembly and stationarity at the limit. The full proof is given in Appendix~\ref{app:kl_proof}.
\end{proof}

\section{Numerical Examples}\label{sec:numerical}
We report three complementary 12-bus-feeder experiments covering discrete identity-resolved assembly, smooth potential tracking with KL diagnostics, and a non-potential saddle map with unequal blocks.\footnote{The experiment code is available at \url{https://github.com/d-vf/Vector_Assembly_D-LI}.}

\subsection{Common Network Setup}\label{subsec:num_setup}
The physical feeder has \(N_{\mathrm{bus}}=12\) buses and \(M=4\) agents at buses \((b_1,b_2,b_3,b_4)=(3,6,9,12)\). The estimator runs over \(N:=M=4\) agents, and \(\dist(j,\ell)\) denotes feeder hop distance. We use \(W=I_M-\alpha L_c\) with \(\alpha=0.249\); for all three experiments, Remark~\ref{rem:subspace} gives \(\sigma=\|\mathbf A_{\mathrm{sys}}\|_2=0.9012\). The feeder edge list and communication-graph construction are given in Appendix~\ref{app:numerical_details}.
\begin{figure}[t]
\centering
\subfloat[Physical feeder.]{%
    \includegraphics[width=0.34\columnwidth]{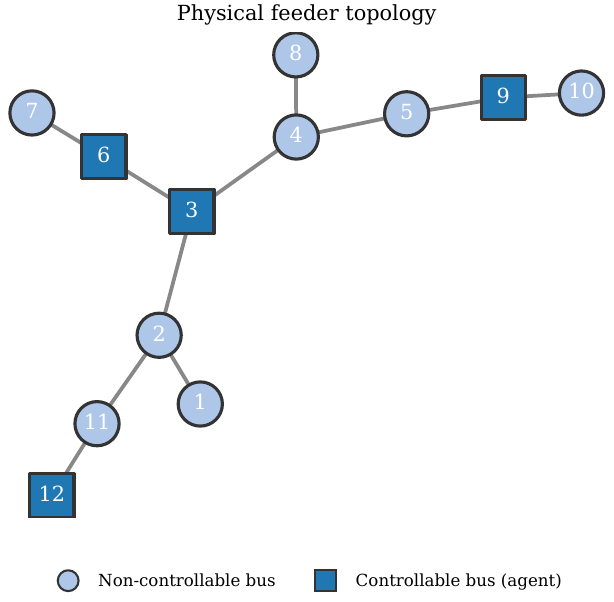}
}\hfill
\subfloat[Communication graph $\mathcal{G}_c$ induced by feeder hop distance.]{%
    \includegraphics[width=0.43\columnwidth]{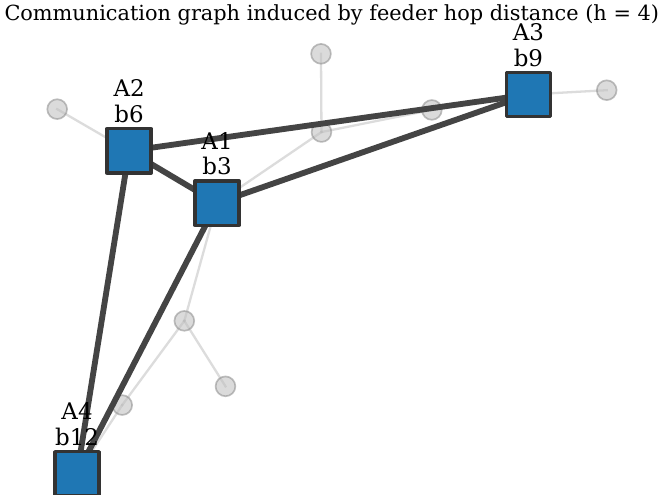}
} 
\caption{Network structure used in the numerical example. The physical topology induces the communication graph that drives diffusion and assembly.}
\label{fig:networks}
\vspace{-1mm}
\end{figure}

\subsection{Experiment A: Binary Demand Satisfaction}\label{subsec:exp_discrete}

\begin{remark}[Scope of Experiment~A]
\label{rem:num_scope_disc}
The D+LI estimator is covered by Theorem~\ref{thm:iss_2ts}, whereas the discrete repair rule lies outside Assumption~\ref{ass:smooth_c} and Theorem~\ref{thm:small_gain_2ts}. Corollary~\ref{cor:tracking} applies with \(\bar\Delta=\sqrt M\); the reported finite-horizon diagnostic uses the observed value \(\max_{0\le k<K}\|\Delta u^k\|_2=\sqrt3\).
\end{remark}

\paragraph{Setup and repair rule.}
Each agent selects \(u_i\in\{0,1\}\), subject to \(h(u):=D-\bar p^\top u\le0\), with \(D=3\), \(\bar p=\mathbf 1_M\), and costs \(\kappa=[1,2,6,12]\). These quantities are common knowledge and fixed offline. Thus, each agent evaluates the same system-wide candidate ranking locally from its assembled estimate and the fixed common data; only the resulting coordinate update is restricted by local write authority. The diagnostic objective is
\begin{align}
J(u):=
\sum_{i=1}^M\kappa_i u_i+\lambda_c[h(u)]_+^2,
\qquad
\lambda_c=50.
\end{align}
The pure-mixing comparison uses the same diffusion operator without selector-partitioned source innovations or an outer repair rule, thereby isolating its fractional mean-field state rather than posing a competing binary optimizer. D+LI thresholds the assembled estimates and applies the simultaneous write-authority-constrained repair. The initialization, run parameters, and complete repair rule are given in Appendix~\ref{app:numerical_details}.

\paragraph{Results.}
Figure~\ref{fig:traj_compare} compares the two information mechanisms. Pure mixing collapses the heterogeneous local values to their fractional mean and therefore loses their coordinate identity. D+LI combines the same mixing layer with selector-partitioned multi-source innovations, allowing every agent to reconstruct the identity-resolved binary vector used by the decentralized repair rule. The averaged state is infeasible, with \(x^\infty=0.472\,\mathbf{1}\) and \([h(x^\infty)]_+=1.112\). In contrast, D+LI assembles identity-resolved estimates and the simultaneous decentralized repair rule recovers the feasible integer solution \(u=[1,1,1,0]\) with \(J(u)=9\) and \([h(u)]_+=0\); see Fig.~\ref{fig:outer_metrics_disc} in Appendix~\ref{app:supplementary_figures}. 
For this illustrative instance, the repair reaches \(u=[1,1,1,0]\) after the first outer update and remains there, so the actual estimator error converges to zero. With \(u^0=\mathbf0\), \(\mathbf s^{0,0}=\mathbf0\), \(K=80\), and \(T=25\), Corollary~\ref{cor:tracking} gives \(\limsup_{k\to\infty}\|\mathbf e^{k,T}\|_2\le3.209\times10^{-1}\), while the calculated finite-horizon bound is \(\max_{0\le k<K}\|\mathbf e^{k,T}\|_2\le2.779\times10^{-1}\).
The finite-horizon unrolling is given in Appendix~\ref{app:tracking_calculations}. Thus Experiment~A illustrates identity-resolved assembly and dissemination; it does not claim that D+LI is necessary to obtain this particular feasible point.
\begin{figure}[t]
\centering
\includegraphics[width=0.64\columnwidth]{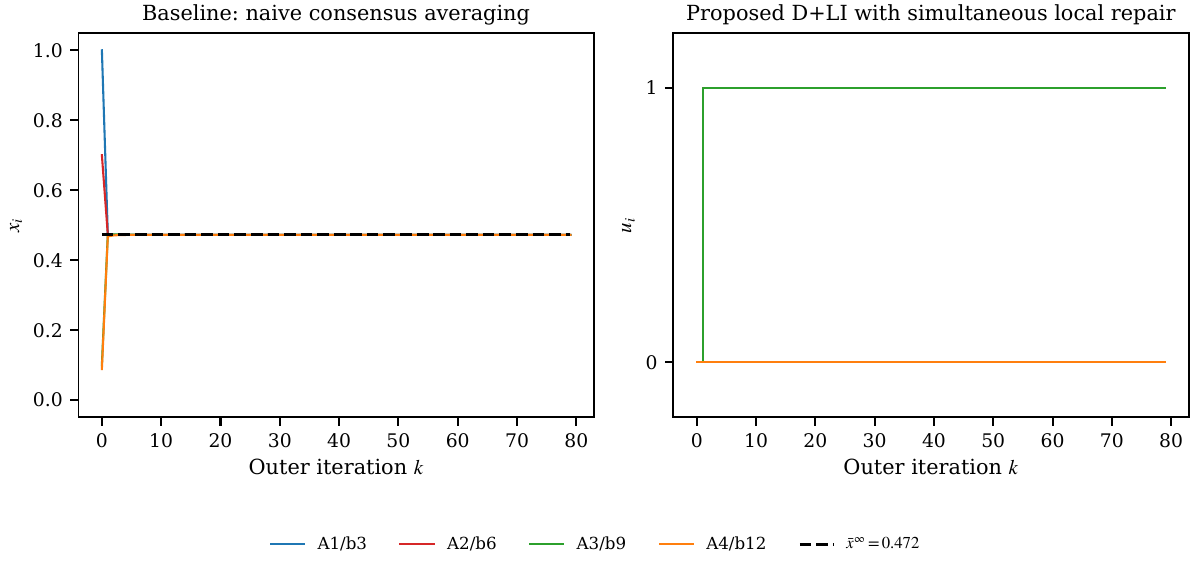}
\vspace{-2mm}
\caption{Experiment~A: decision trajectories. Left: pure mixing converges to a fractional infeasible fixed point. Right: the decentralized repair reaches a feasible integer solution.}
\label{fig:traj_compare}
\vspace{-1mm}
\end{figure}
\subsection{Experiment B: Smooth Voltage Regulation}\label{subsec:exp_smooth}
This experiment holds the smooth projected-gradient outer update fixed and compares D+LI with a pure-mixing baseline. The baseline removes the selector-partitioned source innovations while retaining the same diffusion operator.
With the pure-mixing estimator initialized at zero, its recursion \(\mathbf s^{k,t+1}=(W\otimes I_d)\mathbf s^{k,t}\) receives no block-source innovations and therefore remains identically zero. 
Before gradient evaluation, both methods substitute the exact owned block; in the pure-mixing baseline, all remote blocks remain at their nominal value zero:
\begin{align}
\check s_i^{k,T}
:=
\bar s_i^{k,T}
+\mathbf E_i
\left(
u_i^k-\mathbf E_i^\top\bar s_i^{k,T}
\right).
\label{eq:exact_local_substitution}
\end{align}
Accordingly, for this experiment define \(\hat g^k:=\sum_{i=1}^N\mathbf E_i g_i(\check s_i^{k,T})\). Exact local substitution cannot increase estimation error, preserving the gradient-mismatch bound. Appendix~\ref{app:numerical_details} verifies that all gain,
descent, relative-error, and KL bounds retain the same constants.

\paragraph{Model and objective.}
Voltages are approximated by \(V(u)=V_0+Xu\in\mathbb R^{N_{\mathrm{bus}}}\), with \(X_{ji}=0.4e^{-\dist(j,b_i)/1.5}\), nominal \([V_0]_j=1.02\) p.u., and the stressed buses \(\{3,5,8,10,12\}\) elevated to \(1.12\) p.u. to simulate overvoltage. The safety limit is \(V_{\max}=1.05\) p.u. and \(u_i\in[-1,1]\). The objective is \(J(u)=\tfrac{1}{2}\sum_i \gamma_i u_i^2+\lambda\bigl\|[V(u)-V_{\max}]_+\bigr\|_2^2\), with \(\gamma=[1.0,1.5,0.8,1.2]\) and \(\lambda=80\). For the feeder topology and actuator locations in Fig.~\ref{fig:networks}, the resulting matrix satisfies \(\|X\|_2\approx0.810818\); using its unrounded value gives \(L_{\mathrm{bound}}=106.6882\).
Let
\begin{align}
c_V(u)&:=V_0+Xu-V_{\max}\mathbf1,
&
D_\gamma&:=\diag(\gamma).
\end{align}
The resulting global map is
\begin{align}
g(u)=\nabla J(u)
=
D_\gamma u+2\lambda X^\top[c_V(u)]_+.
\end{align}

By nonexpansiveness of the positive-part mapping, a valid conservative choice is \(L_g=L_{\mathrm{blk}}=L_{\mathrm{bound}}=106.6882\). Moreover, \(J\) and \(\mathcal U=[-1,1]^M\) are semi-algebraic, so Assumptions~\ref{ass:smooth_c} and~\ref{ass:kl} hold with \(F\equiv J\). The Lipschitz calculation is given in Appendix~\ref{app:numerical_details}.

\paragraph{Parameters.}
For the main tracking run, \(K=80\), \(T=10\), \(\eta=0.003\), \(u^0=\mathbf0\), and \(\mathbf s^{0,0}=\mathbf0\). The self-consistent gain condition is strictly satisfied with \(q_T=0.580<1\). Over the reported trajectory, the finite-horizon a posteriori bound is \(\|\mathbf e^{k,T}\|_2\le3.89\times10^{-2}\) for \(0\le k<K\).
The finite-horizon derivation is reported in Appendix~\ref{app:tracking_calculations}; it is not used as an independently established infinite-horizon variation bound. The dominant-mode tightness diagnostic appears in Fig.~\ref{fig:iss_smooth} of Appendix~\ref{app:supplementary_figures}. 
For the finite-length diagnostic, we keep the same outer problem and stepsize but use \(T=15\). With \(L_g=L_{\mathrm{blk}}=106.6882\), \(\sigma=0.901227\), and \((\varepsilon,\tau,\kappa_s)=(5.0\times10^{-3},0.5,10)\), the pointwise KL margins are \(\tilde c_u=119.989>0\) and \(c_s^{+}(T)=1.072>0\).
Thus \eqref{eq:pointwise_absorb} holds strictly. The complete parameter calculation and the auxiliary summed-absorption diagnostic are reported in Appendix~\ref{app:numerical_details}. The finite-length diagnostic in Fig.~\ref{fig:finite_length_smooth} of Appendix~\ref{app:supplementary_figures} uses \(\mathcal L_K:=\sum_{k=0}^{K-1}\|\xi^{k+1}-\xi^k\|_2\), where \(\xi^k=(u^k,\mathbf s^{k,0})\). At \(K=80\), \(\mathcal L_{80}=9.2581\times10^{-1}\) and \(\|\xi^{80}-\xi^{79}\|_2=1.9357\times10^{-3}\).
\paragraph{Results.}
Dominant-mode initialization attains \(\|\mathbf e^T\|_2/\|\mathbf e^0\|_2=\sigma^T\); the positive margins at \(T=15\) certify an admissible KL timescale. For the \(T=10\) tracking run at \(K=80\), the pure-mixing baseline reaches \(J\approx0.3393\) with zero residual violation, whereas D+LI reaches \(J\approx0.1216\) with maximum residual violation \(2.37\times10^{-2}\) p.u. Under the chosen soft-penalty tuning, identity-resolved assembly reduces the objective while permitting a positive residual violation. The observed tail estimator error satisfies \(\max_{40\le k<K}\|\mathbf e^{k,T}\|_2=1.98\times10^{-3}\), below the full-horizon bound.

\subsection{Experiment C: Non-Potential Outer Map}\label{subsec:exp_saddle}
Dualizing the voltage constraint gives, for \(z=(u,\mu)\), the non-potential saddle map

\begin{align}
g(z)
=
\begin{bmatrix}
D_\gamma u+X^\top\mu\\[1mm]
-\bigl(V_0+Xu-V_{\max}\mathbf 1\bigr)
\end{bmatrix}.
\end{align}

Use \(\mathcal U=[-1,1]^4\times[0,50]^{12}\) and \(L_{\mathrm{blk}}=\|J_g\|_2=1.719239\). Each agent applies \eqref{eq:exact_local_substitution} to its owned block of \(z\). The nonzero skew-symmetric part of \(J_g\) excludes a potential. With unequal block dimensions \((d_1,\dots,d_4)=(6,3,4,3)\), \(T=15\), and \(\eta=0.2\), Theorem~\ref{thm:small_gain_2ts} gives \(q_T=0.355<1\). Starting from zero decisions and replicas, over \(K=1500\) iterations the maximum assembly error is \(1.314\times10^{-2}\), below the finite-horizon bound \(1.675\times10^{-2}\).
The construction, parameter checks, and diagnostics are reported in Appendix~\ref{app:experiment_c_details}.

\section{Conclusion}
We established four properties of Selector-Anchored D+LI: the impossibility of isolating the assembled replica by pure averaging with raw additive selector sources; global inner contraction with a dimension-independent factor; finite-$T$ ISS tracking bounds for Lipschitz coordination maps; and KL finite-length convergence of coupled projected-gradient dynamics in the potential case. Write-authority anchoring preserves block identity, enabling coupled-gradient evaluation from assembled estimates with explicit tracking guarantees. 

\appendices

\section{Comparison of Distributed Architectures}
\label{app:comparison}
The comparison in Table~\ref{tab:comparison} distinguishes the target, information structure, and outer-map requirements of D+LI from related distributed methods.

\begin{table*}[!t]
\centering
\caption{Qualitative Comparison of Distributed Methods and Information-Exchange Architectures for Physical Systems}
\label{tab:comparison}
\vspace{-3mm}
{\tiny
\setlength{\tabcolsep}{0.7pt}
\renewcommand{\arraystretch}{1.12}
\sloppy
\begin{tabular}{@{}
>{\raggedright\arraybackslash}p{1.65cm}
>{\raggedright\arraybackslash}p{2.60cm}
>{\raggedright\arraybackslash}p{2.60cm}
>{\raggedright\arraybackslash}p{2.60cm}
>{\raggedright\arraybackslash}p{2.60cm}
>{\raggedright\arraybackslash}p{2.60cm}
>{\raggedright\arraybackslash}p{2.60cm}
@{}}
\toprule
\textbf{Feature} &
\textbf{\shortstack[l]{Proposed (Selector-Anchored\\D+LI)}} &
\textbf{\shortstack[l]{C+I\\(Kar \& Moura)
\cite{kar_distributed_2012, kar_consensus_2013}}} &
\textbf{\shortstack[l]{GT / Dynamic\\consensus
\cite{nedic_achieving_2017,kia_tutorial_2019,shi_extra_2015}}} &
\textbf{\shortstack[l]{Partial-decision\\NE/GNE
\cite{pavel_distributed_2020,bianchi_fully_2021,bianchi_fast_2022}}} &
\textbf{\shortstack[l]{Feedback optimization\\
\cite{picallo_closing_2020}}} &
\textbf{\shortstack[l]{ADMM / dual decomp.\\
\cite{boyd_distributed_2011,kekatos_distributed_2013}}} \\
\midrule

\textbf{Primary Goal} &
Identity-resolved vector assembly &
Estimation and filtering &
Consensus optimization with shared \(x\) / tracking of aggregates &
Equilibrium seeking under partial-decision information &
Regulation &
Constraint splitting \\
\addlinespace

\textbf{Atomic Mechanism} &
Block selectors \(\mathbf E_i\) and coordinate projectors
\(\mathbf P_i=\mathbf E_i\mathbf E_i^\top\) &
Weighted averaging + measurement innovation &
Mixing + tracking state (gradient tracker / dynamic average) &
Local joint-profile estimates + consensus and
pseudo-gradient/proximal updates &
Measurement feedback &
Dual updates / augmented Lagrangian \\
\addlinespace

\textbf{Information structure supported} &
Identity-resolved coupled updates with blockwise write authority &
Measurement-driven distributed estimation &
Shared-variable optimization / aggregate tracking &
Identity-resolved estimates embedded within equilibrium seeking &
Measurement-based regulation &
Constraint splitting through local subproblems and dual variables \\
\addlinespace

\textbf{Coupling Handling} &
Explicit through local replicas of the stacked vector &
Common vector inferred from distributed local observation models &
Implicit through shared decision \(x\) and tracked aggregates &
Explicit through local estimates of the joint decision profile &
Implicit through physical-system feedback &
Implicit through dual variables \\
\addlinespace

\textbf{Innovation Role} &
Deterministic anchoring of locally owned blocks &
Measurement-based correction &
Aggregate correction in the tracking update &
Own-decision update and local-estimate correction &
Integral or feedback correction &
N/A \\
\addlinespace

\textbf{\shortstack[l]{Finite-iteration\\inner accuracy}} &
Explicit ISS and self-consistent finite-\(T\) gain bounds &
Typically asymptotic or steady-state tracking error &
Finite rounds yield residual tracking error depending on
connectivity and stepsizes &
Usually no separate inner loop; estimates and decisions are
analyzed jointly &
N/A (no inner assembly loop) &
Inexact solves require inexact-ADMM conditions
\cite{boyd_distributed_2011,goldstein_fast_2014,boley_local_2013} \\
\addlinespace

\textbf{Outer-map requirement} &
ISS: bounded variation; gain: Lipschitz \(g\), compact \(\mathcal U\);
KL convergence: potential and KL regularity &
Linear or statistical measurement model &
Objective regularity and stepsize--connectivity conditions;
dynamic consensus additionally requires signal-variation conditions &
Monotonicity, cocoercivity, or related problem-specific
operator conditions &
Model and sensitivity regularity of the physical map &
Convexity and separability across the selected split \\
\addlinespace

\textbf{How coupled information enters local updates} &
Assembled estimate of \(u\); feasibility depends on the outer loop &
Local measurement innovations and neighbor exchange toward a
common parameter estimate &
Agreement on shared \(x\) or tracking of network aggregates &
Estimated joint profile enters the local pseudo-gradient or
proximal response &
Physical measurements enter the regulation update &
Constraint information enters through dual updates \\
\bottomrule
\end{tabular}
}
\end{table*}

\newpage

\section{Proofs of Estimator and Gain Results}
\label{app:estimator_gain_proofs}

\subsection{Proof of Lemma~\ref{lem:contraction}}
\begin{proof}
Since \(L=L^\top\) and \(\mathbf P=\mathbf P^\top\), the matrix \(\mathbf A_{\mathrm{sys}}=(W\otimes I_d)-\mathbf P\) is symmetric, so \(\|\mathbf A_{\mathrm{sys}}\|_2=\max\{|\lambda_{\max}|,|\lambda_{\min}|\}\). Under \(0<\alpha<1/\lambda_N(L)\), the eigenvalues of \(W=I_N-\alpha L\) lie in \((0,1]\), and \(W\) has a simple unit eigenvalue with eigenspace \(\operatorname{span}\{\mathbf 1_N\}\). Consequently, \(W\otimes I_d\) has unit eigenspace
\begin{align}
\{\mathbf 1_N\otimes v:\,v\in\mathbb R^d\}.
\end{align}
For any \(\mathbf x\neq 0\),
\begin{align}
\frac{\mathbf x^\top\mathbf A_{\mathrm{sys}}\mathbf x}{\|\mathbf x\|_2^2}
&=
\frac{\mathbf x^\top(W\otimes I_d)\mathbf x}{\|\mathbf x\|_2^2}
-
\frac{\mathbf x^\top\mathbf P\mathbf x}{\|\mathbf x\|_2^2}
\le 1.
\end{align}
Equality would require \(\mathbf x=\mathbf 1_N\otimes v\) and \(\mathbf P\mathbf x=0\). By \eqref{eq:selector_partition}, however,
\begin{align}
(\mathbf 1_N\otimes v)^\top \mathbf P (\mathbf 1_N\otimes v)
=
v^\top\Big(\sum_{i=1}^N\mathbf P_i\Big)v
=
\|v\|_2^2>0
\end{align}
for \(v\neq 0\). Hence \(\lambda_{\max}(\mathbf A_{\mathrm{sys}})<1\). Also, since \(0\preceq \mathbf P\preceq I_{Nd}\),
\begin{align}
\mathbf A_{\mathrm{sys}}
\succeq
(W\otimes I_d)-I_{Nd},
\end{align}
and therefore
\begin{align}
\lambda_{\min}(\mathbf A_{\mathrm{sys}})
\ge
\lambda_{\min}(W\otimes I_d)-1
=
-\alpha\lambda_N(L)>-1.
\end{align}
Thus \(\|\mathbf A_{\mathrm{sys}}\|_2<1\). The associated affine map is therefore contractive. Since \(W\mathbf1_N=\mathbf1_N\),
\begin{align}
(W\otimes I_d)(\mathbf1_N\otimes u^k)
=
\mathbf1_N\otimes u^k,
\end{align}
and its correction term vanishes at \(\mathbf s=\mathbf1_N\otimes u^k\):
\begin{align}
\mathbf P\bigl((\mathbf1_N\otimes u^k)
-(\mathbf1_N\otimes u^k)\bigr)=0.
\end{align}
Thus \(\mathbf1_N\otimes u^k\) is a fixed point and, by contraction, the unique fixed point.
\end{proof}

\subsection{Proof of Proposition~\ref{prop:no_avg}}
\begin{proof}
At any fixed point,
\begin{align}
\bigl(I_{Nd}-W\otimes I_d\bigr)\mathbf s
=
\mathbf b(u).
\end{align}
Since \(\mathbf1_N^\top L=0\),
\begin{align}
(\mathbf1_N^\top\otimes I_d)
\bigl(I_{Nd}-W\otimes I_d\bigr)=0.
\end{align}
Therefore, solvability requires
\begin{align}
0
=
(\mathbf1_N^\top\otimes I_d)\mathbf b(u)
=
\sum_{i=1}^N\mathbf E_i u_i
=
u.
\end{align}
Hence no fixed point exists when \(u\neq0\). When \(u=0\), \(\mathbf b(u)=0\), and connectivity gives the stated consensus fixed-point set.
\end{proof}

\subsection{Coordinate-subspace spectrum}
Let \(o(\ell)\) denote the agent owning coordinate \(\ell\), and let \(\zeta_\ell\in\mathbb R^d\) denote the \(\ell\)-th standard basis vector. Each coordinate subspace \(\{z\otimes \zeta_\ell:z\in\mathbb R^N\}\) is invariant under \(\mathbf A_{\mathrm{sys}}\), on which the operator acts as \(W-\chi_{o(\ell)}\chi_{o(\ell)}^\top\). Consequently,
\begin{align}
\operatorname{spec}(\mathbf A_{\mathrm{sys}})
=
\bigcup_{\ell=1}^{d}
\operatorname{spec}\!\left(
W-\chi_{o(\ell)}\chi_{o(\ell)}^\top
\right),
\end{align}
which yields the expression for \(\sigma\) in Remark~\ref{rem:subspace}.

\subsection{Block-to-global Lipschitz bound}
By orthogonality of the selector blocks,
\begin{align}
\|g(x)-g(y)\|_2^2
&=
\Big\|\sum_{i=1}^N
\mathbf E_i\bigl(g_i(x)-g_i(y)\bigr)\Big\|_2^2
\notag\\
&=
\sum_{i=1}^N
\|g_i(x)-g_i(y)\|_2^2
\notag\\
&\le
\left(\sum_{i=1}^N L_i^2\right)
\|x-y\|_2^2.
\end{align}
Hence
\begin{align}
\operatorname{Lip}_{\mathcal U}(g)
\le
\bar L_g
:=
\left(\sum_{i=1}^N L_i^2\right)^{1/2}
\le
\sqrt N\,L_{\mathrm{blk}}.
\end{align}
Thus \(\bar L_g\) is an admissible generic bound, while \(L_g\) may denote any sharper certified global Lipschitz constant.

\subsection{Proof of Theorem~\ref{thm:iss_2ts}}
\begin{proof}
\eqref{eq:inner_contr_T} follows by iterating \eqref{eq:error_dynamics_inner} and using $\|\mathbf{A}_{\mathrm{sys}}\|_2=\sigma<1$. For \eqref{eq:outer_iss}, expand using the warm start \eqref{eq:warmstart}:
\begin{align}
\mathbf{e}^{k+1,0}
&= \mathbf{s}^{k+1,0} - (\mathbf{1}_N\otimes u^{k+1}) \notag\\
&= \mathbf{s}^{k,T}   - (\mathbf{1}_N\otimes u^{k+1}) \notag\\
&= \underbrace{\mathbf{s}^{k,T}-(\mathbf{1}_N\otimes u^k)}_{\mathbf{e}^{k,T}}
   -(\mathbf{1}_N\otimes\Delta u^k).
\end{align}
Taking norms and using $\|\mathbf{1}_N\otimes\Delta u^k\|_2=\sqrt{N}\|\Delta u^k\|_2$ and \eqref{eq:inner_contr_T} gives \eqref{eq:outer_iss}.
\end{proof}

\subsection{Proof of Theorem~\ref{thm:small_gain_2ts}}
\begin{proof}
By orthogonality of the selector ranges,
\begin{align}
\|\hat g^{k}-g(u^{k})\|_2^2
&=
\sum_{i=1}^N
\|g_i(\bar s_i^{k,T})-g_i(u^{k})\|_2^2
\notag\\
&\le
\sum_{i=1}^N
L_i^2\|\bar s_i^{k,T}-u^{k}\|_2^2
\notag\\
&\le
L_{\mathrm{blk}}^2\|\mathbf e^{k,T}\|_2^2.
\end{align}
The last inequality uses \(u^k\in\mathcal U\) and nonexpansiveness of projection:
\begin{align}
\|\bar s_i^{k,T}-u^k\|_2
\le
\|s_i^{k,T}-u^k\|_2.
\end{align}
Moreover, since \(u^k=\proj_{\mathcal U}(u^k)\), the decentralized update gives
\begin{align}
\|\Delta u^k\|_2
&\le
\|\tilde u^{k+1}-u^k\|_2
=
\eta\|\hat g^k\|_2
\notag\\
&\le
\eta\bigl(\|g(u^k)\|_2
+L_{\mathrm{blk}}\|\mathbf e^{k,T}\|_2\bigr),
\end{align}
which proves \eqref{eq:outer_gain_bounds}.

From Theorem~\ref{thm:iss_2ts},
\begin{align}
\|\mathbf e^{k+1,0}\|_2
\le
\sigma^T\|\mathbf e^{k,0}\|_2+\sqrt N\,\|\Delta u^k\|_2.
\end{align}
Using \eqref{eq:outer_gain_bounds}, the bound \(\|g(u^k)\|_2\le G_{\max}\), and \eqref{eq:inner_contr_T},
\begin{align}
\|\Delta u^k\|_2
&\le
\eta\|\hat g^k\|_2 \notag\\
&\le
\eta\big(G_{\max}+L_{\mathrm{blk}}\|\mathbf e^{k,T}\|_2\big) \notag\\
&\le
\eta\big(G_{\max}+L_{\mathrm{blk}}\sigma^T\|\mathbf e^{k,0}\|_2\big).
\end{align}
Combining the two inequalities gives
\begin{align}
\|\mathbf e^{k+1,0}\|_2
\le
\sigma^T(1+\sqrt N\,\eta L_{\mathrm{blk}})\|\mathbf e^{k,0}\|_2
+\sqrt N\,\eta G_{\max}.
\end{align}
Thus the affine recursion is contractive whenever
\begin{align}
q_T
:=
\sigma^T
\bigl(1+\sqrt N\,\eta L_{\mathrm{blk}}\bigr)
<1,
\end{align}
which is equivalent to \eqref{eq:small_gain_2ts}. Iterating the recursion gives
\begin{align}
\|\mathbf e^{k,0}\|_2
\le{}&
q_T^k\|\mathbf e^{0,0}\|_2
\notag\\
&+
\frac{
\sqrt N\,\eta G_{\max}
\bigl(1-q_T^k\bigr)
}{
1-q_T
}.
\end{align}
Taking the limit superior proves
\begin{align}
\limsup_{k\to\infty}\|\mathbf e^{k,0}\|_2
\le
\frac{\sqrt N\,\eta G_{\max}}{1-q_T}.
\end{align}
Finally, since
\begin{align}
\|\mathbf e^{k,T}\|_2
\le
\sigma^T\|\mathbf e^{k,0}\|_2,
\end{align}
the preceding finite-\(k\) estimate gives
\begin{align}
\|\mathbf e^{k,T}\|_2
\le{}&
\sigma^T q_T^k
\|\mathbf e^{0,0}\|_2
\notag\\
&+
\frac{
\sqrt N\,\eta G_{\max}\sigma^T
\bigl(1-q_T^k\bigr)
}{
1-q_T
}.
\label{eq:small_gain_terminal}
\end{align}
Taking the limit superior gives \eqref{eq:small_gain_terminal_limsup}.
\end{proof}

\subsection{Proof of Corollary~\ref{cor:tracking}}
\begin{proof}
Let \(a:=\sigma^T\in[0,1)\) and \(x_k:=\|\mathbf e^{k,0}\|_2\). By Theorem~\ref{thm:iss_2ts},
\begin{align}
x_{k+1}
\le
a x_k+\sqrt N\,\bar\Delta.
\end{align}
Iteration gives
\begin{align}
x_k
\le
a^k x_0
+
\sqrt N\,\bar\Delta
\sum_{j=0}^{k-1}a^j.
\end{align}
Therefore,
\begin{align}
\limsup_{k\to\infty}x_k
\le
\frac{\sqrt N}{1-\sigma^T}\bar\Delta.
\end{align}
Finally, \(\|\mathbf e^{k,T}\|_2\le\sigma^T x_k\), which gives the second bound.
\end{proof}

\section{Proofs of Descent Results}
\label{app:descent_proofs}

\subsection{Estimator-increment bound}
Since \(\mathbf e^{k,T}=A_T\mathbf e^{k,0}\) and \(\Delta\mathbf s^k=(A_T-I)\mathbf e^{k,0}\), the matrix \(I-A_T\) is nonsingular and
\begin{align}
\mathbf e^{k,T}
=
-A_T(I-A_T)^{-1}\Delta\mathbf s^k .
\end{align}
Moreover, \(A_T\) is symmetric and \(\operatorname{spec}(A_T)\subset[-a,a]\), so
\begin{align}
\|A_T(I-A_T)^{-1}\|_2
&=
\max_{\lambda\in\operatorname{spec}(A_T)}
\frac{|\lambda|}{1-\lambda}
\notag\\
&\le
\frac{a}{1-a}.
\end{align}
This proves \eqref{eq:eT_deltas_bound}.

\subsection{Proof of Theorem~\ref{thm:suffdec_conf}}
\begin{proof}
Recall that \(A_T:=\mathbf A_{\mathrm{sys}}^{\,T}\) denotes the \(T\)-th matrix power. Write
\begin{align}
\Psi(\xi^{k+1})-\Psi(\xi^k)
&= \big(F(u^{k+1})-F(u^k)\big) \notag\\
&\quad +\frac{\kappa_s}{2}\big(\|r^{k+1}\|_2^2-\|r^k\|_2^2\big),
\end{align}
since \(u^k,u^{k+1}\in\mathcal U\).
For the \(F\)-term, projection optimality of \(u^{k+1}=\proj_{\mathcal U}(u^k-\eta\hat g^k)\) gives
\begin{align}
\left\langle
\eta^{-1}(u^k-u^{k+1})-\hat g^k,\,
u-u^{k+1}
\right\rangle
\le 0,
\qquad \forall u\in\mathcal U .
\end{align}
Taking \(u=u^k\) yields
\begin{align}
\langle \hat g^k,\Delta u^k\rangle
\le
-\eta^{-1}\|\Delta u^k\|_2^2 .
\end{align}
Since \(\hat g^k=\nabla F(u^k)+\delta^k\), the descent lemma gives
\begin{align}
F(u^{k+1})-F(u^k)
\le
-\Bigl(\eta^{-1}-\frac{L_g}{2}\Bigr)\|\Delta u^k\|_2^2
-\langle \delta^k,\Delta u^k\rangle .
\end{align}
Using Young's inequality and
\(\|\delta^k\|_2\le L_{\mathrm{blk}}\|\mathbf e^{k,T}\|_2\),
\begin{align}
-\langle \delta^k,\Delta u^k\rangle
\le
C_u\|\Delta u^k\|_2^2
+
C_e\|\mathbf e^{k,T}\|_2^2 ,
\end{align}
with \(C_u=1/(2\varepsilon)\) and
\(C_e=(\varepsilon/2)L_{\mathrm{blk}}^2\). Hence
\begin{align}
F(u^{k+1})-F(u^k)
\le
-\Big(\eta^{-1}-\frac{L_g}{2}-C_u\Big)
\|\Delta u^k\|_2^2
+
C_e\|\mathbf e^{k,T}\|_2^2 .
\end{align}
For the tracking term, using \(r^{k+1}=\mathbf e^{k,T}-\mathbf H\Delta u^k\), \(\mathbf e^{k,T}=A_T\mathbf e^{k,0}\), symmetry of \(A_T\), and \(\|A_T\|_2=\sigma^T<1\), we have \(\mathrm{spec}(A_T)\subset[-\sigma^T,\sigma^T]\). Hence, for every eigenvalue \(\lambda\) of \(A_T\),
\begin{align}
1-\lambda^2
=
(1-\lambda)(1+\lambda)
\ge
\frac{1-\sigma^T}{1+\sigma^T}(1-\lambda)^2,
\end{align}
so
\begin{align}
I-A_T^2
\succeq
\frac{1-\sigma^T}{1+\sigma^T}(I-A_T)^2 .
\end{align}
Moreover,
\begin{align}
\|r^{k+1}\|_2^2-\|r^k\|_2^2
={}&
-(\mathbf e^{k,0})^\top
(I-A_T^2)\mathbf e^{k,0}
\notag\\
&-2\langle\mathbf e^{k,T},
\mathbf H\Delta u^k\rangle
+N\|\Delta u^k\|_2^2 .
\end{align}
Using
\begin{align}
\Delta\mathbf s^k
=-(I-A_T)\mathbf e^{k,0}
\end{align}
and Young's inequality,
\begin{align}
-2\langle\mathbf e^{k,T},\mathbf H\Delta u^k\rangle
\le
\tau\|\mathbf e^{k,T}\|_2^2
+\tau^{-1}N\|\Delta u^k\|_2^2,
\end{align}
we obtain
\begin{align}
\|r^{k+1}\|_2^2-\|r^k\|_2^2
&\le -\frac{1-\sigma^T}{1+\sigma^T}
\|\Delta\mathbf s^k\|_2^2 \notag\\
&\quad +\tau\|\mathbf e^{k,T}\|_2^2
+\bigl(1+\tau^{-1}\bigr)N\|\Delta u^k\|_2^2 .
\end{align}
Collecting terms gives
\begin{align}
\Psi(\xi^{k+1})
&\le \Psi(\xi^k)
-\tilde c_u\|\Delta u^k\|_2^2 \notag\\
&\quad -c_s\|\Delta\mathbf s^k\|_2^2
+\tilde C_e\|\mathbf e^{k,T}\|_2^2,
\end{align}
with the constants defined in Theorem~\ref{thm:suffdec_conf}.
\end{proof}

\subsection{Proof of Corollary~\ref{cor:pointwise_descent}}
\begin{proof}
By \eqref{eq:eT_deltas_bound},
\begin{align}
\tilde C_e\|\mathbf e^{k,T}\|_2^2
\le
\tilde C_e
\left(\frac{\sigma^T}{1-\sigma^T}\right)^2
\|\Delta\mathbf s^k\|_2^2.
\end{align}
Substituting this bound into Theorem~\ref{thm:suffdec_conf} gives the result.
\end{proof}

\section{Estimator $\ell_2$-gain}
\begin{lemma}[Estimator $\ell_2$-gain]\label{lem:l2_gain}
Let \(a:=\sigma^T\in[0,1)\). Under the hypotheses of Theorem~\ref{thm:iss_2ts},
\begin{align}
\|\mathbf e^{k+1,0}\|_2
&\le
a\|\mathbf e^{k,0}\|_2+\sqrt N\|\Delta u^k\|_2, \notag\\
\|\mathbf e^{k,T}\|_2
&\le
a\|\mathbf e^{k,0}\|_2 .
\end{align}
Then, for any \(\nu>0\) and any \(K\ge 1\),
\begin{align}
\sum_{k=0}^{K-1}\|\mathbf e^{k,T}\|_2^2
&\le
(1+\nu)\frac{a^2}{1-a^2}\|\mathbf e^{0,0}\|_2^2 \notag\\
&\quad +
\Gamma_e(T,\nu)
\sum_{k=0}^{K-1}\|\Delta u^k\|_2^2,
\end{align}
where
\begin{align}
\Gamma_e(T,\nu)
:=
(1+\nu^{-1})
\frac{N a^2}{(1-a)^2}
=
(1+\nu^{-1})
\frac{N\sigma^{2T}}{(1-\sigma^T)^2}.
\end{align}
\end{lemma}

\begin{proof}
Define
\begin{equation}
x_k:=\|\mathbf e^{k,0}\|_2,
\qquad
d_k:=\sqrt N\,\|\Delta u^k\|_2 .
\end{equation}
The estimator recursion gives
\begin{equation}
x_{k+1}\le a x_k+d_k .
\end{equation}
Iterating this inequality yields, for every \(k\ge 0\),
\begin{equation}
x_k
\le
a^k x_0
+
\sum_{j=0}^{k-1}a^{k-1-j}d_j,
\end{equation}
where the sum is understood to be zero when \(k=0\). Since \(\|\mathbf e^{k,T}\|_2\le a x_k\), it follows that
\begin{equation}
\|\mathbf e^{k,T}\|_2
\le
a^{k+1}x_0
+
\sum_{j=0}^{k-1}a^{k-j}d_j .
\end{equation}
For any \(\nu>0\), Young's inequality gives
\begin{align}
\|\mathbf e^{k,T}\|_2^2
\le{}&
(1+\nu)a^{2k+2}x_0^2 \notag\\
&+
(1+\nu^{-1})
\left(
\sum_{j=0}^{k-1}a^{k-j}d_j
\right)^2 .
\end{align}
Summing the first term over \(k=0,\dots,K-1\) gives
\begin{equation}
\sum_{k=0}^{K-1}a^{2k+2}x_0^2
\le
\frac{a^2}{1-a^2}x_0^2 .
\end{equation}

For the second term, consider the causal convolution kernel \(h_\ell:=a^\ell\), \(\ell\ge1\). Its \(\ell_1\)-norm is
\begin{equation}
\|h\|_{\ell_1}
=
\sum_{\ell=1}^{\infty}a^\ell
=
\frac{a}{1-a}.
\end{equation}
Young's convolution inequality therefore gives
\begin{align}
\sum_{k=0}^{K-1}
\left(
\sum_{j=0}^{k-1}a^{k-j}d_j
\right)^2
&\le
\left(\frac{a}{1-a}\right)^2
\sum_{k=0}^{K-1}d_k^2 \notag\\
&=
\frac{Na^2}{(1-a)^2}
\sum_{k=0}^{K-1}\|\Delta u^k\|_2^2 .
\end{align}
Combining the preceding bounds yields
\begin{align}
\sum_{k=0}^{K-1}\|\mathbf e^{k,T}\|_2^2
\le{}&
(1+\nu)\frac{a^2}{1-a^2}
\|\mathbf e^{0,0}\|_2^2 \notag\\
&+
(1+\nu^{-1})
\frac{Na^2}{(1-a)^2}
\sum_{k=0}^{K-1}\|\Delta u^k\|_2^2,
\end{align}
which is the claimed estimate.
\end{proof}

\section{Summed absorption}

\begin{corollary}[Summed absorption]\label{cor:absorption}
Under the hypotheses and notation of Theorem~\ref{thm:suffdec_conf}, if, for some \(\nu>0\),
\begin{equation}\label{eq:absorb_cond}
\tilde C_e\,\Gamma_e(T,\nu) < \tilde c_u,
\end{equation}
then for every \(K\ge 1\),
\begin{align}
&\sum_{k=0}^{K-1}\Big(
(\tilde c_u-\tilde C_e\Gamma_e(T,\nu))\|\Delta u^k\|_2^2
+c_s\|\Delta\mathbf s^k\|_2^2
\Big) \notag\\
&\quad\le
\Psi(\xi^0)-\Psi(\xi^K)
+\tilde C_e(1+\nu)
\frac{\sigma^{2T}}{1-\sigma^{2T}}
\|\mathbf e^{0,0}\|_2^2 .
\end{align}

For \(0<\sigma<1\), the summed-absorption condition \eqref{eq:absorb_cond} holds for all \(T\ge T_{\min}^{\mathrm{abs}}(\nu)\), where
\begin{equation}\label{eq:T_min_abs}
T_{\min}^{\mathrm{abs}}(\nu)
:=
\left\lfloor
\frac{
\log\!\bigl(
1+\sqrt{(1+\nu^{-1})N\tilde C_e/\tilde c_u}
\bigr)
}{
\log(1/\sigma)
}
\right\rfloor
+1 .
\end{equation}
If \(\sigma=0\), then \(\Gamma_e(T,\nu)=0\) for every \(T\ge1\); hence the condition holds whenever \(\tilde c_u>0\), and one may set \(T_{\min}^{\mathrm{abs}}(\nu)=1\).
\end{corollary}

\begin{proof}
Summing Theorem~\ref{thm:suffdec_conf} over \(k=0,\dots,K-1\) gives
\begin{align}
\Psi(\xi^K)-\Psi(\xi^0)
&\le
-\tilde c_u\sum_{k=0}^{K-1}\|\Delta u^k\|_2^2
\notag\\
&\quad
-c_s\sum_{k=0}^{K-1}\|\Delta\mathbf s^k\|_2^2
+\tilde C_e\sum_{k=0}^{K-1}\|\mathbf e^{k,T}\|_2^2 .
\end{align}
Applying Lemma~\ref{lem:l2_gain},
\begin{align}
\sum_{k=0}^{K-1}\|\mathbf e^{k,T}\|_2^2
&\le
(1+\nu)\frac{\sigma^{2T}}{1-\sigma^{2T}}\|\mathbf e^{0,0}\|_2^2 \notag\\
&\quad +\Gamma_e(T,\nu)\sum_{k=0}^{K-1}\|\Delta u^k\|_2^2.
\end{align}
Substitution yields
\begin{align}
&\sum_{k=0}^{K-1}\Big(
(\tilde c_u-\tilde C_e\Gamma_e(T,\nu))\|\Delta u^k\|_2^2
+c_s\|\Delta\mathbf s^k\|_2^2
\Big) \notag\\
&\quad\le
\Psi(\xi^0)-\Psi(\xi^K)
+\tilde C_e(1+\nu)\frac{\sigma^{2T}}{1-\sigma^{2T}}\|\mathbf e^{0,0}\|_2^2.
\end{align}
Thus \(\tilde C_e\Gamma_e(T,\nu)<\tilde c_u\) is sufficient for absorption. Solving this inequality for \(T\) gives \eqref{eq:T_min_abs}.
\end{proof}

\section{Proof of the Relative-Error Lemma
(Lemma~\ref{lem:relerr_conf})}
\label{app:relerr_proof}
\begin{proof}
Projection optimality gives
\begin{align}
\eta^{-1}(u^k-u^{k+1})-\hat g^k \in N_{\mathcal U}(u^{k+1}).
\end{align}
Define
\begin{align}
w_u^{k+1}
&:= \nabla F(u^{k+1})
+\kappa_s\mathbf H^\top(\mathbf H u^{k+1}-\mathbf s^{k+1,0}) \notag\\
&\quad +\eta^{-1}(u^k-u^{k+1})-\hat g^k,
\end{align}
and
\begin{align}
w_s^{k+1}:=\kappa_s(\mathbf s^{k+1,0}-\mathbf H u^{k+1}).
\end{align}
Then \((w_u^{k+1},w_s^{k+1})\in\partial\Psi(\xi^{k+1})\). Using
\(\hat g^k=\nabla F(u^k)+\delta^k\), \(L_g\)-Lipschitz continuity of \(\nabla F\),
\(\|\delta^k\|_2\le L_{\mathrm{blk}}\|\mathbf e^{k,T}\|_2\), and
\(\|\mathbf H^\top z\|_2\le \sqrt N\|z\|_2\), we obtain
\begin{align}
\|w_u^{k+1}\|_2
&\le (\eta^{-1}+L_g)\|\Delta u^k\|_2
+L_{\mathrm{blk}}\|\mathbf e^{k,T}\|_2 \notag\\
&\quad +\kappa_s\sqrt N
\|\mathbf H u^{k+1}-\mathbf s^{k+1,0}\|_2 .
\end{align}

Also,
\begin{align}
\|\mathbf H u^{k+1}-\mathbf s^{k+1,0}\|_2
&= \|\mathbf H\Delta u^k-\mathbf e^{k,T}\|_2 \notag\\
&\le \sqrt N\,\|\Delta u^k\|_2+\|\mathbf e^{k,T}\|_2.
\end{align}
Using \eqref{eq:eT_deltas_bound}, we further have
\begin{align}
\|\mathbf e^{k,T}\|_2
\le
\frac{\sigma^T}{1-\sigma^T}
\|\mathbf{s}^{k+1,0}-\mathbf{s}^{k,0}\|_2 .
\end{align}
Combining this with the analogous bound for \(w_s^{k+1}\) yields
\begin{align}
\dist(0,\partial\Psi(\xi^{k+1}))
\le
C\Big(
\|u^{k+1}-u^k\|_2
+
\|\mathbf{s}^{k+1,0}-\mathbf{s}^{k,0}\|_2
\Big),
\end{align}
for some finite constant \(C>0\), after absorbing fixed problem-dependent constants.
\end{proof}

\section{Proof of KL Convergence in the Potential Case
(Theorem~\ref{thm:kl_conf})}
\label{app:kl_proof}
\begin{proof}
By Corollary~\ref{cor:pointwise_descent}, the sequence satisfies the sufficient decrease condition
\begin{align}
\Psi(\xi^{k+1})
\le
\Psi(\xi^k)
-
\mu\|\xi^{k+1}-\xi^k\|_2^2,
\end{align}
where
\begin{align}
\mu
:=
\min\left\{
\tilde c_u,\,
c_s-\tilde C_e
\left(\frac{\sigma^T}{1-\sigma^T}\right)^2
\right\}
>0 .
\end{align}
By Lemma~\ref{lem:relerr_conf}, it also satisfies the relative-error condition
\begin{align}
\dist(0,\partial\Psi(\xi^{k+1}))
\le
C\|\xi^{k+1}-\xi^k\|_2 .
\end{align}
Since \(\Psi\) is proper, lower semicontinuous, semi-algebraic, and bounded below on its domain, it satisfies the KL property. Moreover, \(\{u^k\}\subset\mathcal U\) is bounded, and sufficient decrease implies that \(\{\Psi(\xi^k)\}\) is nonincreasing and convergent. Since \(F\) is bounded below on the compact set \(\mathcal U\), and
\(\Psi(\xi^k)\le\Psi(\xi^0)\), we have
\begin{equation*}
\frac{\kappa_s}{2}
\|\mathbf s^{k,0}-\mathbf H u^k\|_2^2
\le
\Psi(\xi^0)-\inf_{u\in\mathcal U}F(u).
\end{equation*}
Therefore, the tracking term is bounded; together with \(\{u^k\}\subset\mathcal U\), this implies that \(\{\xi^k\}\) is bounded.

Let \(\{\xi^{k_j}\}\) be any convergent subsequence, with \(\xi^{k_j}\to\bar\xi=(\bar u,\bar{\mathbf s})\). Since \(u^{k_j}\in\mathcal U\) and \(\mathcal U\) is closed, we have \(\bar u\in\mathcal U\). Therefore, \(\iota_{\mathcal U}(u^{k_j})= \iota_{\mathcal U}(\bar u)=0\). By continuity of \(F\) and of the quadratic tracking term,
\begin{align}
\Psi(\xi^{k_j})\to\Psi(\bar\xi).
\end{align}
Thus the subsequence-continuity condition required by the standard KL convergence theorem is satisfied. The standard KL convergence argument~\cite{attouch_convergence_2013, bolte_clarke_2007} for bounded sequences satisfying sufficient decrease and relative error therefore gives finite length,
\begin{align}
\sum_{k=0}^\infty \|\xi^{k+1}-\xi^k\|_2<\infty,
\end{align}
and convergence of \(\xi^k\) to a critical point
\(\xi^\star=(u^\star,\mathbf s^\star)\) of \(\Psi\).

It remains to identify the critical point. Since
\begin{align}
\|\mathbf e^{k,T}\|_2
\le
\frac{\sigma^T}{1-\sigma^T}
\|\mathbf s^{k+1,0}-\mathbf s^{k,0}\|_2
\end{align}
and finite length implies
\(\|\mathbf s^{k+1,0}-\mathbf s^{k,0}\|_2\to0\), we have
\(\|\mathbf e^{k,T}\|_2\to0\). Also
\(\|\Delta u^k\|_2\to0\). Hence
\begin{align}
\|\mathbf s^{k+1,0}-(\mathbf 1_N\otimes u^{k+1})\|_2
\le
\|\mathbf e^{k,T}\|_2+\sqrt N\|\Delta u^k\|_2
\to0 .
\end{align}
Therefore
\begin{align}
\mathbf s^\star=\mathbf 1_N\otimes u^\star .
\end{align}
The \(u\)-component of
\(0\in\partial\Psi(\xi^\star)\) then reduces to
\begin{align}
0\in\nabla F(u^\star)+N_{\mathcal U}(u^\star)
=
g(u^\star)+N_{\mathcal U}(u^\star),
\end{align}
which proves stationarity.
\end{proof}

\section{Numerical Implementation Details}
\label{app:numerical_details}

\subsection{Network construction details}
The controllable buses \(\mathcal V_c=\{3,6,9,12\}\) host the four agents, with assignment \((b_1,b_2,b_3,b_4)=(3,6,9,12)\). The radial feeder has edges
\begin{align}
&(1,2),(2,3),(3,4),(4,5),(3,6),(6,7),\notag\\
&(4,8),(5,9),(9,10),(2,11),(11,12).
\end{align}
The communication graph \(\mathcal G_c\) connects agents within feeder hop distance \(h_c=4\), yielding \(|\mathcal E_c|=5\), \(d_{\max}=3\), and \(\lambda_{\max}(L_c)=4\). Thus \(W=I_M-\alpha L_c\) with \(\alpha=0.249\) satisfies \(\alpha<1/\lambda_{\max}(L_c)=0.25\) and \(\alpha\le1/d_{\max}\).

\subsection{Experiment A implementation and repair rule}
For the reported run, \(K=80\), \(T=25\), \(u^0=\mathbf0\), and \(\mathbf s^{0,0}=\mathbf1_M\otimes u^0=\mathbf0\). The pure-mixing comparison is initialized at \(x^0=[1,\,0.7,\,0.1,\,0.088]^\top\), whose average is \(0.472\), and applies
\begin{align}
x^{k+1}=Wx^k
\end{align}
on \([0,1]^M\), without selector-partitioned source innovations or an outer repair rule. The initialization is arbitrary; any interior initialization produces a fractional mean-field fixed point. After each inner loop, agent \(i\) forms a binary proxy by thresholding each coordinate of its assembled estimate at \(1/2\):
\begin{align}
\tilde u_j^{(i),k}
:=
\mathds{1}\!\left\{
[s_i^{k,T}]_j\ge \frac12
\right\},
\qquad j=1,\ldots,M.
\end{align}
Agent \(i\) then enforces write-authority consistency on its own coordinate by defining
\begin{align}
\hat u_j^{(i),k}
:=
\begin{cases}
u_i^k, & j=i,\\
\tilde u_j^{(i),k}, & j\neq i.
\end{cases}
\end{align}
Agent \(i\) computes its perceived shortfall
\begin{align}
R_i^k
:=
\left[D-\mathbf 1^\top\hat u^{(i),k}\right]_+,
\end{align}
which is integer-valued because \(D\) and \(\hat u^{(i),k}\) are integer-valued. Let
\begin{align}
\mathcal O_i^k
:=
\left\{
j:\hat u_j^{(i),k}=0
\right\},
\end{align}
and let \(\mathcal A_i^k\subseteq\mathcal O_i^k\) contain the \(R_i^k\) lowest-cost indices in \(\mathcal O_i^k\), with ties resolved by increasing agent index. Each agent changes only its own decision according to
\begin{align}
u_i^{k+1}
=
\begin{cases}
1,
& u_i^k=1\ \text{or}\ i\in\mathcal A_i^k,\\
0,
& \text{otherwise}.
\end{cases}
\label{eq:binary_repair}
\end{align}
All agents apply \eqref{eq:binary_repair} simultaneously. Thus, agent \(i\) may identify several indices as candidates for activation, but it has write authority only over its own coordinate \(u_i\).

\subsection{Exact local substitution in Experiment B}
Since \(\mathbf E_i^\top u^k=u_i^k\), \eqref{eq:exact_local_substitution} gives
\begin{align}
\check s_i^{k,T}-u^k
=
(I_d-\mathbf P_i)
(\bar s_i^{k,T}-u^k).
\end{align}
Because \(\mathcal U=\prod_i\mathcal U_i\), \(\check s_i^{k,T}\in\mathcal U\), and
\begin{align}
\|\check s_i^{k,T}-u^k\|_2
&\le
\|\bar s_i^{k,T}-u^k\|_2
\le
\|s_i^{k,T}-u^k\|_2.
\label{eq:exact_local_error_bound}
\end{align}
Therefore,
\begin{align}
\|g_i(\check s_i^{k,T})-g_i(u^k)\|_2
\le
L_i\|s_i^{k,T}-u^k\|_2.
\end{align}
\begin{remark}[Exact local substitution]
\label{rem:exact_local_substitution}
The estimator--controller analysis uses the gradient-evaluation point only through the blockwise mismatch estimate
\begin{align}
\|g_i(y_i^{k,T})-g_i(u^k)\|_2
\le
L_i\|s_i^{k,T}-u^k\|_2.
\end{align}
By \eqref{eq:exact_local_error_bound}, this estimate holds with \(y_i^{k,T}=\check s_i^{k,T}\). Therefore, Theorems~\ref{thm:small_gain_2ts} and~\ref{thm:suffdec_conf}, Corollary~\ref{cor:pointwise_descent}, Lemma~\ref{lem:relerr_conf}, and Theorem~\ref{thm:kl_conf} remain valid, with the same constants, when the projected estimate \(\bar s_i^{k,T}\) is replaced at gradient evaluation by \(\check s_i^{k,T}\).
\end{remark}

\subsection{Lipschitz and KL parameter calculations for Experiment B}
The objective decomposition used in Experiment~B is
\begin{align}
f_i(u)=\frac{1}{2}\gamma_i u_i^2,
\qquad
c_V(u):=V_0+Xu-V_{\max}\mathbf 1,
\qquad
\rho=\lambda.
\end{align}
Since the positive-part mapping is nonexpansive,
\begin{align}
\|g(x)-g(y)\|_2
&\le
\bigl(\|D_\gamma\|_2+2\lambda\|X\|_2^2\bigr)
\|x-y\|_2 .
\end{align}
Thus a valid global Lipschitz constant is
\begin{align}
L_{\mathrm{bound}}
:=
\|D_\gamma\|_2+2\lambda\|X\|_2^2
=
106.6882.
\end{align}
Moreover, \(g_i=\mathbf E_i^\top g\), so
\begin{align}
\|g_i(x)-g_i(y)\|_2
\le L_{\mathrm{bound}}\|x-y\|_2.
\end{align}
Thus choosing \(L_i:=L_{\mathrm{bound}}\) for every block is valid and yields \(L_{\mathrm{blk}}=L_{\mathrm{bound}}\). For the finite-length diagnostic, \(T=15\), \(\sigma^T=0.2101\), and \((\varepsilon,\tau,\kappa_s)=(5.0\times10^{-3},0.5,10)\). The resulting constants are
\begin{align}
C_u&=100.000,
&
C_e&=28.456,
&
\tilde C_e&=30.956,
\notag\\
\tilde c_u&=119.989>0,
&
c_s&=3.263,
&
c_s^{+}(T)&=1.072>0.
\end{align}
Hence the pointwise condition \eqref{eq:pointwise_absorb} holds strictly. For the auxiliary summed-absorption estimate of Corollary~\ref{cor:absorption}, with \(\Gamma_e\) from Lemma~\ref{lem:l2_gain}, take \(\nu=1\). Then
\begin{align}
\Gamma_e(T,1)&=0.566,
\notag\\
\tilde c_u-\tilde C_e\Gamma_e(T,1)
&=102.460>0,
\notag\\
T_{\min}^{\mathrm{abs}}(1)&=9<T=15.
\end{align}

\subsection{Experiment C implementation and diagnostics}\label{app:experiment_c_details}
Experiment~B uses a potential projected-gradient map, whereas Experiment~A uses a discrete repair rule outside the smooth theory; both assign one coordinate per agent. Experiment~C complements them with a smooth non-potential outer map and unequal block dimensions. Dualizing the voltage constraint instead of penalizing it introduces multipliers \(\mu\in\mathbb R^{N_{\mathrm{bus}}}\) and gives, for \(z=(u,\mu)\), the saddle-point coordination map
\begin{align}
g(z)
=
\begin{bmatrix}
D_\gamma u+X^\top\mu\\[1mm]
-\bigl(V_0+Xu-V_{\max}\mathbf 1\bigr)
\end{bmatrix},
\end{align}
on \(\mathcal U=[-1,1]^M\times[0,\bar\mu]^{N_{\mathrm{bus}}}\) with \(\bar\mu=50\). The Jacobian of \(g\) has a nonzero skew-symmetric part, so no potential \(F\) exists: Assumption~\ref{ass:kl} fails and Theorem~\ref{thm:kl_conf} does not apply. Theorem~\ref{thm:small_gain_2ts} does apply, since it requires only Lipschitz continuity of \(g\) on the compact set \(\mathcal U\). Agent \(i\) owns \(u_i\) together with the dual variables of the buses assigned to \(b_i\) by minimum feeder hop distance, with ties resolved by increasing agent index, giving unequal block dimensions \((d_1,\dots,d_4)=(6,3,4,3)\) and \(d=16\). The preceding saddle-map display uses the grouped coordinate order \(z=(u,\mu)\). For the D+LI implementation, a fixed permutation arranges \(z\) into the four agent-owned blocks; this permutation preserves Euclidean norms and the Lipschitz constant. We initialize \(z^0=\mathbf0\) and \(\mathbf s^{0,0}=\mathbf1_M\otimes z^0=\mathbf0\). The communication graph, \(\alpha\), and \(W\) are unchanged from Section~\ref{subsec:num_setup}. Consistent with Remark~\ref{rem:subspace}, the contraction factor is unchanged at \(\sigma=0.901227\) despite the fourfold increase in \(d\) and the unequal blocks. Since \(g\) is affine, \(\|J_g\|_2=1.719239\). Thus choosing \(L_i=L_{\mathrm{blk}}=L_g=1.719239\) is valid. With \(T=15\) and \(\eta=0.2\), the gain condition \eqref{eq:small_gain_2ts} gives \(q_T=0.355<1\). The artificial upper dual cap is inactive throughout the reported trajectory, since 
\begin{align}
\max_{0\le k\le K}\max_j\mu_j^k
=
3.84
<
\bar\mu=50.
\end{align}
Thus, the dual cap serves only to render \(\mathcal U\) compact in this experiment.
Over \(K=1500\) outer iterations,
\begin{align}
\bar\Delta_K
:=
\max_{0\le k<K}\|z^{k+1}-z^k\|_2
=
3.15\times10^{-2}.
\end{align}
Iterating Theorem~\ref{thm:iss_2ts} therefore yields
\(\|\mathbf e^{k,T}\|_2\le1.675\times10^{-2}\) for
\(0\le k<K\).
The observed error satisfies \(\max_{0\le k<K}\|\mathbf e^{k,T}\|_2 = 1.314\times10^{-2}\) (Fig.~\ref{fig:saddle}). 

The final outer iterate reduces the maximum violation to \(2.03\times10^{-3}\) and has quadratic control cost \(\tfrac12 (u^K)^\top D_\gamma u^K=0.139\); we report this trajectory as a diagnostic only, since the present analysis makes no outer-loop convergence claim for non-potential \(g\). The point of the experiment is that the assembly layer retains its finite-\(T\) guarantee under an outer map outside the scope of Assumption~\ref{ass:kl}.

\begin{figure}[!htbp]
\centering
\includegraphics[width=0.75\columnwidth]{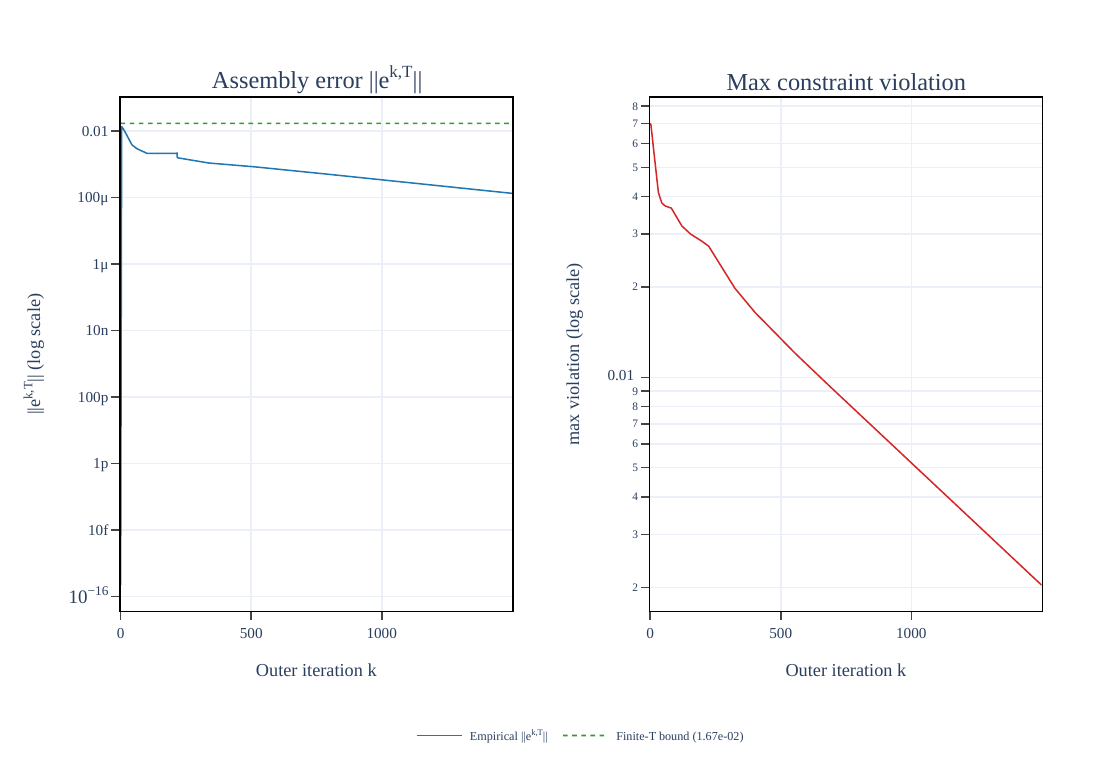}
\vspace{-2mm}
\caption{Experiment~C --- non-potential saddle-point outer map. Left: assembly error \(\|\mathbf e^{k,T}\|_2\) against the finite-horizon bound obtained by iterating Theorem~\ref{thm:iss_2ts}. Right: maximum constraint violation (log scale).}
\label{fig:saddle}
\vspace{-2mm}
\end{figure}

\subsection{Finite-horizon tracking calculations}
\label{app:tracking_calculations}
For Experiment~A, \(M=4\) and the global variation bound \(\bar\Delta=\sqrt M=2\). Corollary~\ref{cor:tracking} therefore gives
\begin{align}
\limsup_{k\to\infty}\|\mathbf e^{k,T}\|_2
\le 3.209\times10^{-1}.
\end{align}
For the reported finite trajectory, let \(a:=\sigma^T\). Unrolling the estimator recursion from \(\mathbf e^{0,0}=\mathbf0\) gives, for every \(0\le k<K\),
\begin{align}
\|\mathbf e^{k,T}\|_2
&\le
\sqrt M\sum_{j=0}^{k-1}
a^{k-j}\|\Delta u^j\|_2
\notag\\
&\le
\frac{\sqrt M\,a}{1-a}
\max_{0\le j<K}\|\Delta u^j\|_2.
\end{align}
Since \(\max_{0\le j<K}\|\Delta u^j\|_2=\sqrt3\),
\begin{align}
\max_{0\le k<K}\|\mathbf e^{k,T}\|_2
\le 2.779\times10^{-1}.
\end{align}

For Experiment~B, the main tracking run uses \(K=80\), \(T=10\), and \(\eta=0.003\). With \(M=4\), \(L_{\mathrm{blk}}\le L_{\mathrm{bound}}=106.6882\), and \(\sigma=0.901227\),
\begin{align}
q_T
&=
\sigma^T\bigl(1+\sqrt M\,\eta L_{\mathrm{bound}}\bigr)
\notag\\
&=0.580<1.
\end{align}
The associated tracking-gain ratio is
\begin{align}
\frac{\sqrt M\,\eta L_{\mathrm{bound}}\sigma^T}
{1-\sigma^T}
=0.350.
\end{align}
Over the simulated horizon, define
\begin{align}
\bar\Delta_K
&:=
\max_{0\le k<K}\|u^{k+1}-u^k\|_2
\notag\\
&=3.56\times10^{-2}.
\end{align}
Since \(\mathbf e^{0,0}=\mathbf0\), iteration of the estimator recursion yields
\begin{align}
\|\mathbf e^{k,T}\|_2
&\le
\frac{\sqrt M\,\sigma^T}{1-\sigma^T}\bar\Delta_K
\notag\\
&=3.89\times10^{-2},
\qquad 0\le k<K.
\end{align}
This is an a posteriori bound for the reported finite trajectory, not an independently established infinite-horizon variation bound. The observed tail error is
\begin{align}
\max_{40\le k<K}\|\mathbf e^{k,T}\|_2
=1.98\times10^{-3}.
\end{align}

\subsection{Supplementary numerical figures}
\label{app:supplementary_figures}

\begin{figure}[!htbp]
\centering
\includegraphics[width=\columnwidth]{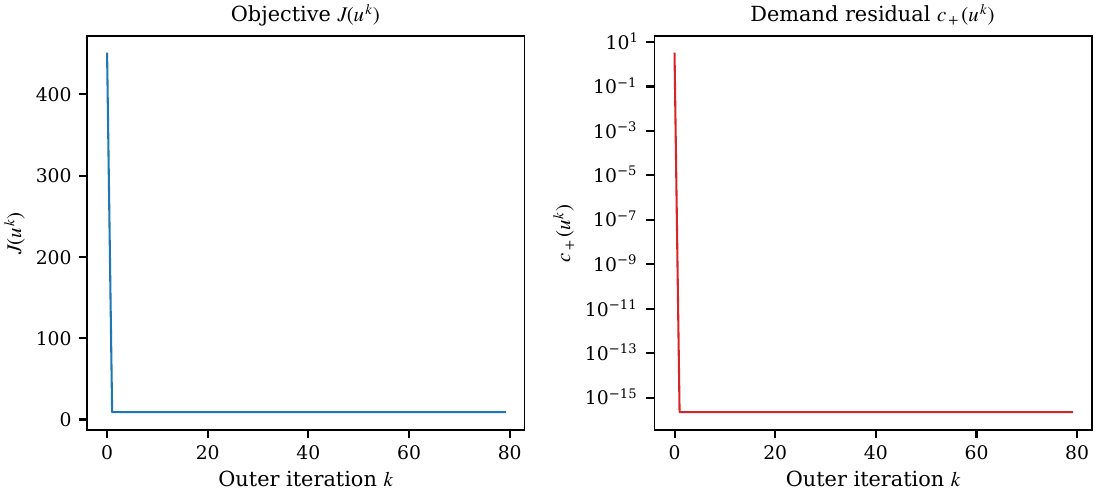}
\caption{Experiment~A: objective and demand residual on a logarithmic scale.}
\label{fig:outer_metrics_disc}
\end{figure}

\begin{figure}[!htbp]
\centering
\includegraphics[width=\columnwidth]{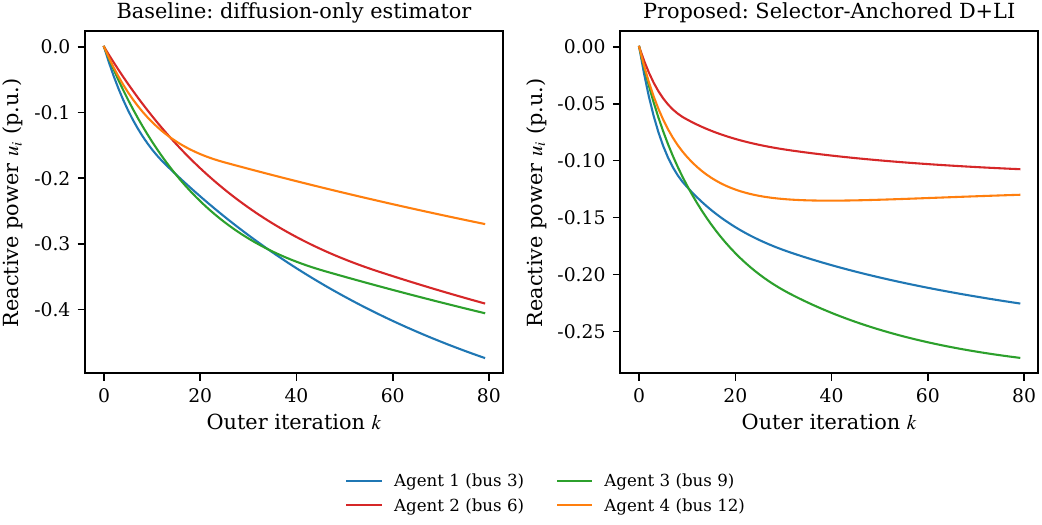}
\caption{Experiment~B: reactive-power trajectories for the pure-mixing baseline and D+LI.}
\label{fig:traj_smooth}
\end{figure}

\begin{figure}[!htbp]
\centering
\includegraphics[width=\columnwidth]{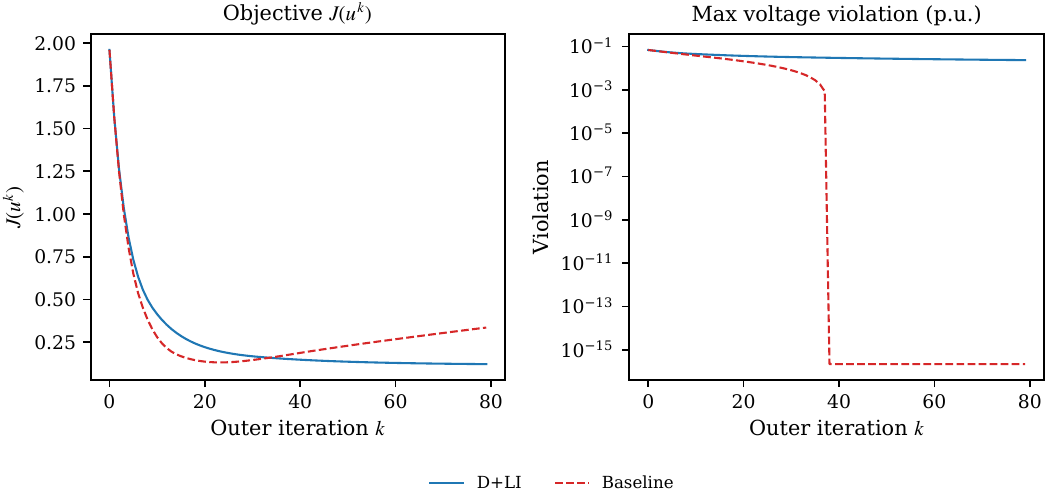}
\caption{Experiment~B: objective and maximum voltage violation on a logarithmic scale.}
\label{fig:outer_metrics_smooth}
\end{figure}

\begin{figure}[!htbp]
\centering
\includegraphics[width=0.70\linewidth]{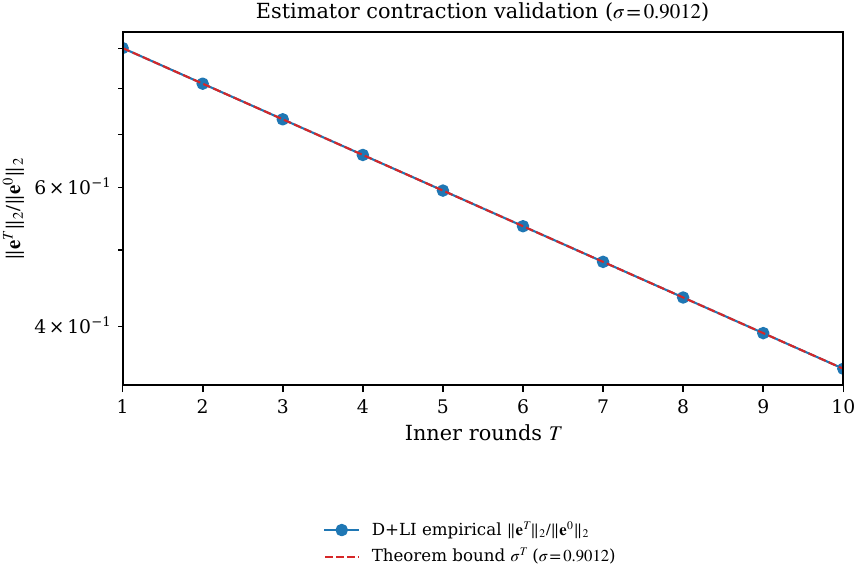}
\caption{Experiment~B: contraction tightness under dominant-mode
initialization. The ratio
\(\|\mathbf e^T\|_2/\|\mathbf e^0\|_2\) attains
\(\sigma^T\), with \(\sigma=0.9012\).}
\label{fig:iss_smooth}
\end{figure}

\begin{figure}[!htbp]
\centering
\includegraphics[width=\columnwidth]{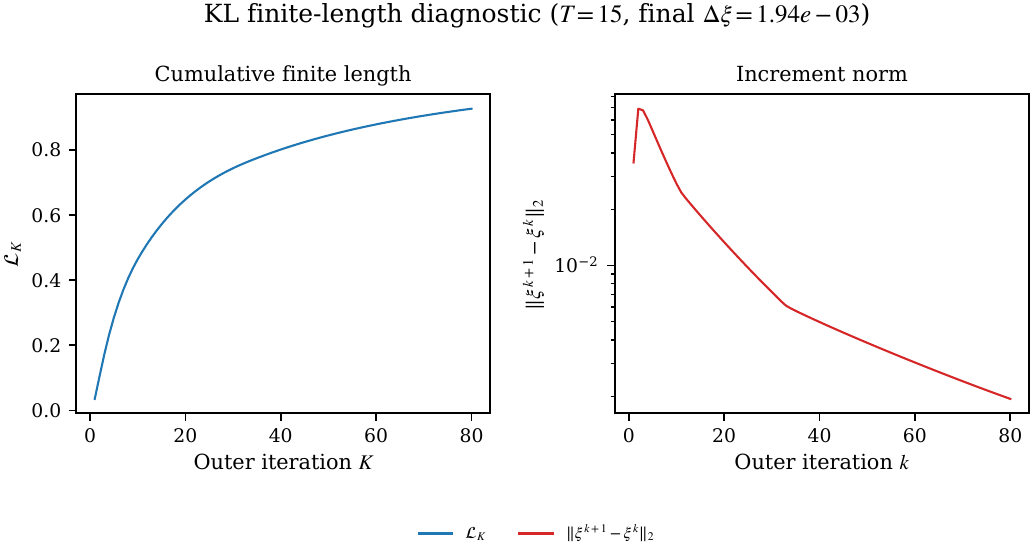}
\caption{Experiment~B: KL finite-length diagnostic for the admissible parameter tuple. The cumulative length \(\mathcal L_K=\sum_{k=0}^{K-1}\|\xi^{k+1}-\xi^k\|_2\) stabilizes while \(\|\xi^{k+1}-\xi^k\|_2\) decays.}
\label{fig:finite_length_smooth}
\end{figure}

\FloatBarrier

\section*{Acknowledgment}
LLM tools (ChatGPT~\cite{openai_chatgpt}, Claude~\cite{anthropic_claude}, and Gemini~\cite{google_gemini}) assisted with review, proofreading, and consistency checks. The authors retain full responsibility for all content and results.

\bibliographystyle{IEEEtran}
\bibliography{references}

\end{document}